\documentclass{amsart}
\usepackage[margin=.8in]{geometry}

\usepackage{amssymb}
\usepackage{amsthm}
\usepackage{amsmath}
\usepackage{tikz}
\usepackage{float}
\usepackage{defs}
\usepackage{subcaption}

\newtheorem{thm}{Theorem}[section]

\newtheorem{lemma}[thm]{Lemma}
\newtheorem{cor}[thm]{Corollary}
\newtheorem{conj}[thm]{Conjecture} 

\theoremstyle{definition}
\newtheorem{definition}[thm]{Definition}

\theoremstyle{remark}
\newtheorem{remark}[thm]{Remark}

\numberwithin{equation}{section}

\begin{document}

\title{Extreme Nonnegative Quadratics over Stanley-Reisner Varieties}

\author{Kevin Shu}
\address{Department of Mathematics, Georgia Institute of Technology, 
Atlanta, GA}
\email{kshu8@gatech.edu}
\urladdr{www.kevinshu.me}

\begin{abstract}
    We consider the convex geometry of the cone of nonnegative quadratics over Stanley-Reisner varieties.
    Stanley-Reisner varieties (which are unions of coordinate planes) are amongst the simplest real projective varieties, so this is potentially a starting point that can generalize to more complicated real projective varieties.
    This subject has some suprising connections to algebraic topology and category theory, which we exploit heavily in our work.

    These questions are also valuable in applied math, because they directly translate to questions about positive semidefinite (PSD) matrices.
    In particular, this relates to a long line of work concerning the extent to which it is possible to approximately check that a matrix is PSD by checking that some principle submatrices are PSD, or to check if a partial matrix can be approximately completed to full PSD matrix.

    We systematize both these practical and theoretical questions using a framework based on algebraic topology, category theory, and convex geometry.
    As applications of this framework we are able to classify the extreme nonnegative quadratics over many Stanley-Reisner varieties.
    We plan to follow these structural results with a paper that is more focused on quantitative questions about PSD matrix completion, which have applications in sparse semidefinite programming.
\end{abstract}
\maketitle

\section{Introduction}

Positive semidefinite (PSD) matrices, which are symmetric matrices whose eigenvalues are all nonnegative, are of fundamental interest in fields ranging from pure mathematics to engineering applications.
It is well known that if a matrix is PSD, then all of its submatrices are PSD, and there has been recent interest in seeking an approximate converse to this fact.
Specifically, the interest is in understanding whether we can assert that a matrix is `close' to being PSD if enough its submatrices are PSD.

Understanding the connection between the eigenvalues of a symmetric matrix to the eigenvalues of its principal submatrices has had a long history, including theorems such as the Cauchy interlacing theorem, and the Schur-Horn theorem.
These theorems form a bridge between convex geometry and spectral properties of matrices, which we seek to extend in this work.
\subsection{Motivation and Existing Results}
We motivate the questions considered in this paper by describing two lines of work from seemingly disparate fields: positive semidefinite matrix completion and the sums-of-squares paradigm in real algebraic geometry.

Early work in this field includes the work of Grone et al. \cite{GRONE1984109} on the PSD matrix completion problem.
In the PSD matrix completion problem, the entries of a symmetric matrix corresponding to the edges of a graph are given, and we want to know if the remaining entries can be chosen to make the resulting matrix PSD.
If the graph is a chordal graph, then a partial matrix can be completed to a PSD matrix if and only if all of the submatrices corresponding to cliques in the graph are PSD.

This work on PSD matrix completion was extended work by Johnson and Laurent \cite{MR1342017, Laurent2009} to consider the instances where the underlying graph is series-parallel using the cycle-completability conditions.

For more general graphs, it is known that this condition that ensuring that all of the submatrices corresponding to cliques are PSD will not guarantee that a partial matrix is PSD completable.
We will consider whether it is possible to conclude that this condition ensures that a partial matrix is `approximately PSD completable'.

We can even ask this approximation question in cases when all of the entries of the matrix are known, and we want to know if ensuring that some submatrices of the matrix are PSD will guarantee that the whole matrix is `approximately' PSD.
Other work in this line includes considering the extent to which it is possible to check that a matrix is PSD by checking that random submatrices  \cite{bakshi2020testing, blekherman2020sparse} are PSD.
A similar question, where we consider symmetric matrices where all $k\times k$ principal submatrices are PSD was considered in \cite{blekherman2020hyperbolic}.
A generalization of this question involving checking PSDness on general linear subspaces, other than coordinate subspaces, was considered in \cite{song2021approximations}, which shows that unless we are in the PSD-matrix completion setting, we would need to consider exponentially many low dimensional projections of a matrix to approximately check if it is PSD.

Relatively recently, this question has been tackled from the perspective of real algebraic geometry, and in particular, to the question of the difference between sums-of-squares and nonnegative quadratic forms on real projective varieties.
For instance, in \cite{free_resolution}, this theorem of Grone was shown to be implied by a theorem of Fr\"{o}berg in \cite{MR1171260} about the Betti numbers of Stanley-Reisner varieties which are chordal.
In \cite{BS2020}, we make this connection more explicit, and describe how the PSD-completable matrices correspond to sums-of-squares quadratic forms, and how matrices with PSD submatrices can be regarded as nonnegative quadratic forms on certain algebraic varieties.
Using these ideas, we make quantitative guarantees about how close a partial matrix where some of the submatrices are PSD is from being PSD completable.

The theory of nonnegative and sums-of-squares polynomials is of fundamental interest in both the theory of real algebraic geometry and in optimization, ever since Hilbert's 1880 result about the existence of nonnegative polynomials which are not sums of squares \cite{MR1510517}.
This connection between real algebraic geometry and positive semidefinite matrices can be thought of as the cornerstone for the theory of Laserre hierarchies for semidefinite programming \cite{SOSChapter}.
In recent study, the theory of sums-of-squares forms and nonnegative polynomials have been extended to quadratic polynomials defined over a real projective variety, as in \cite{Blekherman2015SumsOS}.
Many of earlier results about sums-of-squares over real projective varieties make use of some homological algebraic; we will only make light use of homological algebra in this paper.

\subsection{Our Contributions}
Many of the above results can be framed as questions about the difference between sum-of-squares and nonnegative quadratic forms on Stanley-Reisner varieties of various types.
Most of our results in this paper concern the structure of the convex cone of nonnegative quadratics on general Stanley-Reisner varieties, and in particular, what the extreme rays of these cones are.
These extreme rays will turn out to be related to the geometric structure of the set of subsets of the matrix we demand to be PSD.

We first observe that because being PSD is closed under taking submatrices, we might as well assume that the set of submatrices we take to be PSD is closed under taking subsets.
Such a set, which is downward closed, is called an abstract simplicial complex, and these are fundamental to the study of algebraic topology.

We next observe that there are various maps that can take simplicial complexes to other simplicial complexes.
Then, we apply some category theoretic thinking to show that these maps of simplicial complexes can be systematically turned into maps of the corresponding convex cones. 
In many cases, geometric properties of these simplicial maps translate into geometric properties of the maps of cones.
We consider this especially idea interesting because it might be applicable to the study of nonnegative quadratics over other algebraic varieties.

\begin{figure}[H]
    \centering
    \includegraphics[width=0.7\linewidth]{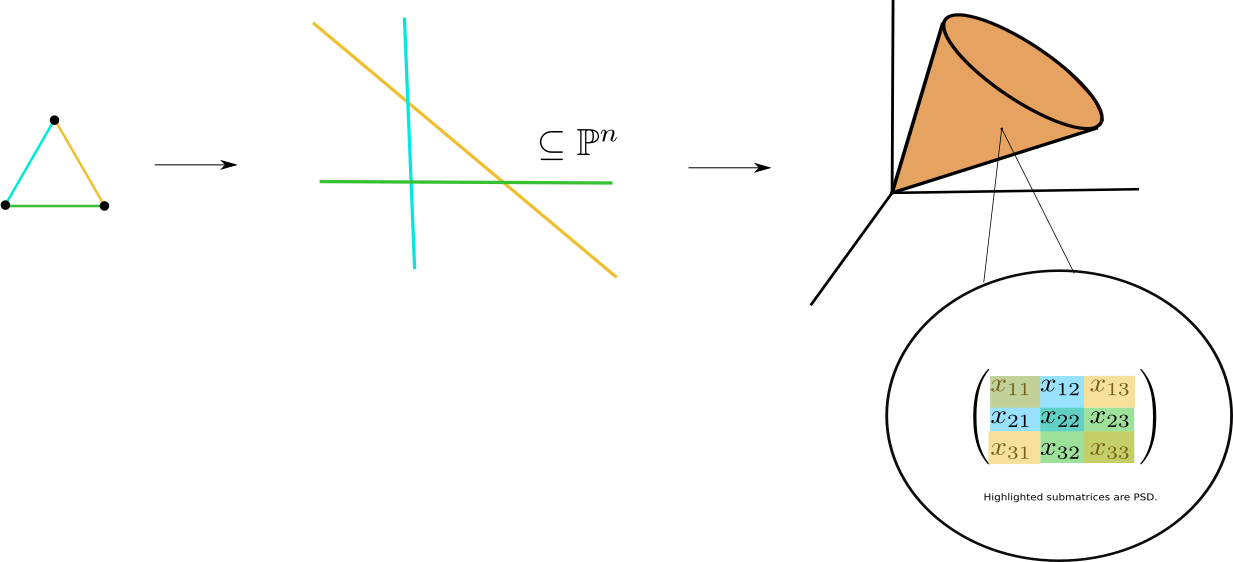}
    \caption{A diagram showing the functors sending a simplicial complex (left) to a Stanley-Reisner variety, and then sending that Stanley-Reisner variety to a convex cone of nonnegative quadratics (right).}
\end{figure}

In particular, it is known that there is a class of simplicial complexes, which we call the chordal complexes, for which nonnegative quadratic forms are all sums-of-squares.
We introduce the idea of expressing a general simplicial complex as a quotient of a chordal complex.
This means that we are able to study the nonnegative quadratic forms associated to a complex in terms of sums-of-squares quadratics on a related chordal complex.

\begin{figure}[H]
    \centering
    \begin{tikzpicture}
        \node[shape=circle,draw=black] (1) at (1,5) {1};
        \node[shape=circle,draw=black] (3) at (2,5) {2};
        \node[shape=circle,draw=black] (5) at (3,5) {3};
        \node (dots) at (3,6) {$\dots$};
        \node[shape=circle,draw=black] (8) at (2,6) {$n$};
        \node[shape=circle,draw=black] (7) at (1,6) {$1^*$};
        \node[shape=circle,draw=black] (9) at (4,5.5) {$4$};

        \path [-](1) edge node[left] {} (3);
        \path [-](3) edge node[left] {} (5);
        \path [-](8) edge node[left] {} (7);
        \path [-](8) edge node[left] {} (dots);
        \path [-](9) edge node[left] {} (dots);
        \path [-](9) edge node[left] {} (5);

        \draw[rounded corners] (0.5, 6.5) rectangle (1.5, 4.5) {};
        \draw [->] (1,4.5) -- (1, 3);

        \node[shape=circle,draw=black] (a1) at (1,2.5) {1};
        \node[shape=circle,draw=black] (a3) at (2,2) {2};
        \node[shape=circle,draw=black] (a5) at (3,2) {3};
        \node (adots) at (3,3) {$\dots$};
        \node[shape=circle,draw=black] (a8) at (2,3) {$n$};
        \node[shape=circle,draw=black] (a9) at (4,2.5) {$4$};

        \path [-](a1) edge node[left] {} (a3);
        \path [-](a3) edge node[left] {} (a5);
        \path [-](a8) edge node[left] {} (a1);
        \path [-](a8) edge node[left] {} (adots);
        \path [-](a9) edge node[left] {} (adots);
        \path [-](a9) edge node[left] {} (a5);
    \end{tikzpicture}
    \caption{An example of a chordal quotient of a path (which is chordal), covering a cycle.}
\end{figure}
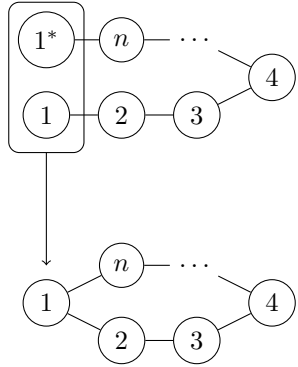

We also introduce a notion of rank for nonnegative quadratic forms over Stanley-Reisner varieties, which we call local rank.
This notion of local rank is slightly different than the usual notion of rank for quadratic forms, but it is useful for characterizing how complicated a given nonnegative quadratic is.
We then try to characterize the local ranks of the extreme nonnegative quadratic forms.

We give both upper bounds and lower bounds on the local ranks of extreme nonnegative quadratic forms, which depend on the geometric properties of these simplicial complexes.

Our upper bounds use the chordal quotient idea, and the fact that we can represent a sum-of-squares quadratic form in coordinates as the matrix of inner products of a collection of vectors in a Hilbert space.
That is, every PSD matrix is a Gram matrix.
This is in some sense a categorification of the notion of the notion of a PSD matrix, which allows us to express basic operations like summation in terms of algebraic operations on vectors.
This additional structure will allow us to find decompositions of a given sum-of-squares quadratic which are compatible with a given simplicial map, which will then allow us to decompose any nonnegative quadratic form on a given complex as the sum of low-rank nonnegative quadratics.

Our lower bounds are inspired by homotopy theory. 
We can construct some interesting extreme rays by first finding a map from a given complex to a complex with an extreme ray of high rank, and then pulling this extreme ray back to our starting complex.
As long as this map satisfies some conditions which we call strong connectivity, this procedure produces high rank extreme rays in our starting complex.

We are interested in the case when the local rank of every extreme ray is 1.
There are a rich collection of simplicial complexes which have extreme local rank 1, and the ones we have found have the property that they are homotopy equivalent to a 1-dimensional complex.
This structure is strong enough that it will allow us to obtain quantitative bounds on how well sums-of-squares quadratic forms approximate nonnegative quadratic forms in general, which we will explore in greater detail in a sequel.

\section{Preliminaries}
\subsection{Introduction to Stanley-Reisner Varieties}
We define abstract simplicial complexes, which are a well known combinatorial model for topological spaces used in algebraic topology.
\begin{definition}[Abstract Simplicial Complex]\label{def:abstract_simplicial_complex}
    If $\groundSet$ is a finite set, then $\ASC \subseteq 2^\groundSet$ is a \textbf{simplicial complex} if for each $S \in \ASC$ and each subset $T \subseteq S$, $T \in \ASC$, and also all single element sets in $2^{\groundSet}$ are in $\ASC$.
    We will generally not be concerned with the distinction between an element $x \in E$ and the single element set $\{x\} \in 2^{E}$.
    
    If $S \in \ASC$, we say that $S$ is a \textbf{face} of $\ASC$ of dimension $|S|-1$, and if $S \in \ASC$ is not strictly contained in any faces of $\ASC$, then we call $S$ a \textbf{facet}. We will also refer to elements of $\groundSet$ as being \textbf{vertices} of $\ASC$.
\end{definition}
Any simplicial complex $\ASC$ has a topological realization $|\ASC|$; see \cite{Hatcher} for details.

Given a simplicial complex $\ASC$, we will associate to it a real projective variety, known as the Stanley-Reisner variety of $\ASC$.
See \cite[Chapter 1]{miller2004combinatorial} for details on Stanley-Reisner rings and varieties.
\textbf{Real projective varieties} can be regarded as a subset of real projective space, $\PROJ^n$, defined by the simultaneous vanishing of a collection of homogeneous polynomial equations, or else, in terms of spectra of a graded ring (see \cite{hartshorne2013algebraic}).
 We will denote the projetive variety defined by a homogeneous ideal $\mathcal{I}$ as $\mathcal{V}(\mathcal{I})$.
To define the Stanley-Reisner variety of $\ASC$, we will first define its Stanley-Reisner ideal.
\begin{definition}[Stanley-Reisner Ideal]\label{def:stanley_reisner_ideal}
    If $\ASC$ is a simplicial complex over a ground set $\groundSet$, then the (real) \textbf{Stanley-Reisner ideal} of $\ASC$ is an ideal of $\R[x_i : i \in \groundSet]$ given by
    \[
        \SRI(\ASC) = \langle \prod_{i \in S} x_i : S \subseteq \groundSet, S \not \in \ASC \rangle.
    \]
\end{definition}
The Stanley-Reisner variety is the vanishing locus of the Stanley-Reisner ideal.
\begin{definition}[Stanley-Reisner Variety]\label{def:stanley_reisner_variety}
    If $\ASC$ is a simplicial complex over a ground set $\groundSet$, then the (real) \textbf{Stanley-Reisner variety} of $\ASC$ is a projective variety contained in $\PROJ^{|\groundSet|}$ where
    \[
        \SRV(\ASC) = \mathcal{V}(\SRI(\ASC)) = \{x \in \PROJ^n : \forall f \in \SRI(\ASC), f(x) = 0\}.
    \]
\end{definition}

Equivalently, the Stanley-Reisner variety of $\ASC$ is the union of coordinate planes, which correspond to faces of $\ASC$.
\begin{lemma}\label{lem:coordinate_planes}
\[
    \SRV(\ASC) = \bigcup_{S \in \ASC} \SPAN(\{e_i : i \in S\}),
\]
where $e_i$ denotes the standard basis vector in $\R^{\groundSet}$, and $\SPAN$ denotes the linear span of a set of vectors.
\end{lemma}

Associated to any real projective variety is its graded coordinate ring, whose elements are polynomial functions on the variety.
\begin{definition}[Coordinate Ring]\label{def:coordinate_ring}
    Let $V$ be  a real projective variety, then we will denote the (graded) coordinate ring of $V$ by
    \[
        \CoorR(V) = \R[x_1 ,\dots, x_n] / \IDE(V).
    \]
    We will use $\CoorR_d(V)$ to denote the vector space of degree $d$ homogeneous polynomials in $\CoorR(V)$.
\end{definition}
\subsection{Nonnegative and Sum-of-Squares Quadratics on Stanley Reisner Varieties}
Inside the coordinate ring of any real projective variety, there are two important subsets: the nonnegative polynomials and the sum-of-squares polynomials.
\begin{definition}[Nonnegative Polynomials]\label{def:nonnegative_quadratic}
    If $V$ is a real projective variety, then a homogeneous polynomial of even degree $p \in \CoorR(V)$ is nonnegative if for every point $v \in V$, $p(v) \ge 0$ (while the value $p(v)$ is not technically well defined on projective space, this property is still defined as long as $p$ has even degree).
\end{definition}
\begin{definition}[Sum-of-Squares Polynomials]\label{def:sums_of_squares}
    If $V$ is a real projective variety, then a homogeneous polynomial $p \in \CoorR_{2d}(V)$ is sum-of-squares there exist polynomials $a_1, a_2 ,\dots, a_k \in \CoorR_d(V)$ so that $p = a_1^2 + a_2^2 + \dots + a_k^2$.
\end{definition}

In this paper, we will mostly be interested in nonnegative and sum-of-squares quadratic forms, which lie in $\CoorR_2(V)$.
Some of the ideas described here can be generalized to higher degree polynomials, but our main theorems only apply to quadratic forms.

We will use the notation $\POS(V)$ to denote the set of nonnegative quadratics on $V$, and the notation $\SOS(V)$ to denote the set of sum-of-squares quadratic forms.

$\POS(V)$ and $\SOS(V)$ both have the structure of \textbf{convex cones}; they are closed under addition and nonnegative scalar multiplication. 
If $C$ is a convex cone, $x \in C$ is said to span an \textbf{extreme ray} if there do not exist linearly independent $z, y \in C$ so that $x = z + y$.
More generally, we say that $S \subseteq C$ is a \textbf{face} of $C$ if it is a convex cone, and for any $x \in S$, if $a,b \in C$ so that $a + b = x$, then $a, b \in S$.

It is easy to see that the extreme rays of $\SOS(V)$ are precisely the squares of linear polynomials in $V$. 
Most of this paper will be concerned with the structure of the extreme rays and faces of $\POS(V)$ when $V$ is a Stanley-Reisner variety.

\subsection{Functoriality}
We will make light use of category theory in this paper, as we can express our idea of using combinatorial maps of simplicial complexes to describe linear maps of convex cones using functors.
This idea was described in other language in \cite{laurent2012gram}, in the case when the underlying combinatorial map is a contraction of a simplicial complex along an edge.

So far, we have described three important constructions: the Stanley-Reisner variety of a simplicial complex; the cone of nonnegative quadratics of an algebraic variety, and the cone of sum-of-squares quadratics on an algebraic variety.
We now upgrade these constructions to be functorial.
First, we need to introduce the categories we are working with.

\begin{definition}[$\simp$]\label{def:simp_cat}
    The category of simplicial complexes, denoted $\simp$, has simplicial complexes as objects, and simplicial maps as morphisms.
    A \textbf{simplicial map} $\phi : \ASC \rightarrow \ASCb$ is a function from the ground set of $\ASC$ to the ground set of $\ASCb$ so that the image under $f$ of any face of $\ASC$ is a face of $\ASCb$.
\end{definition}

\begin{definition}[$\lin$]\label{def:lin_cat}
    The category of projective varieties with linear maps, denoted $\lin$, has real projective varieties as objects and linear morphisms of projective varities as morphisms.
\end{definition}

\begin{definition}[$\cone$]\label{def:cone_cat}
    The category of convex cones, denoted $\cone$, has convex cones as objects, and linear maps between convex cones as morphisms.
\end{definition}

It is clear that the operation $\mathcal{V}$ takes objects in $\simp$ to objects in $\lin$, and the operations $\POS$ and $\SOS$ take objects in $\lin$ to objects in $\cone$.
We now describe how these operations act on morphisms.

Given a simplicial map $\phi : \ASC \rightarrow \ASCb$, we let
\[
    \mathcal{V}(\phi) : \mathcal{V}(\ASC) \rightarrow \mathcal{V}(\ASCb)
\]
be the linear map of varieties which sends the coordinate vector $e_i$ to the coordinate vector $e_{\phi(i)}$.
It is easy to see from lemma \ref{lem:coordinate_planes} that $\mathcal{V}(\phi)$ in fact takes $\mathcal{V}(\ASC)$ to $\mathcal{V}(\ASCb)$, so that $\mathcal{V}$ is in fact a covariant functor.

Given a linear map between real projective varieties, $\phi : V_1 \rightarrow V_2$, we can define a map of convex cones $\POS(\phi) : \POS(V_2) \rightarrow \POS(V_1)$, given by pulling back nonnegative quadratic forms on $V_2$ to $V_1$ by $\phi$.
That is, we send the quadratic form $q \in \POS(V_2)$ to the quadratic form whose value at $x \in V_1$ is $q(\phi(x))$.
It is clear then that the pullback of a nonnegative quadratic form is a nonnegative quadratic form.
So, $\POS$ can be regarded as a contravariant functor from $\lin$ to $\cone$.

For any simplicial complex $\ASC$, we will write $\POS(\ASC)$ as shorthand for $\POS(\mathcal{V}(\ASC))$. 
For a map of simplicial complexes $\phi$, we will denote by $\phi^*$ the linear map $\POS(\mathcal{V}(\phi))$.

Similarly, the $\SOS$ construction can also be thought of as a contravariant functor from $\lin$ to $\cone$. 
Our aim at this point is to descibe natural maps of simplicial complexes which give rise to nice maps of convex cones.

\subsection{Coordinates for $\CoorR_2(\mathcal{V}(\ASC))$ and Partial Matrices}
We will work a lot in coordinates for this paper; in particular, we will need to exploit the connection between quadratic forms over $\mathcal{V}(\ASC)$ and partial matrices.

To define these coordinates, we will first need a combinatorial definition,
\begin{definition}[$k$-Skeleton]\label{def:k_skeleton}
    \textbf{The $k$ skeleton} of an abstract simplicial complex $\ASC$ is the abstract simplicial complex defined by 
    \[
        \skel k \ASC = \{S \in \ASC : |S| \le k+1\}.
    \]
    The \textbf{strict} $k$ skeleton of an abstract simplicial complex $\ASC$ is defined as 
    \[
        \stskel k \ASC = \{S \in \ASC : |S| = k+1\}.
    \]
\end{definition}

We can see from the definition that $\CoorR_2(\mathcal{V}(\ASC))$ is spanned by the monomials $x_ix_j$, where $\{i,j\} \in \skel 1 \ASC$.
For any $q \in \CoorR_2(\mathcal{V}(\ASC))$, there is a unique way of writing
\[
    q = \sum_{\{i,j\} \in \skel 1 \ASC} X_{i,j} x_ix_j,
\]
and conversely, for any collection of coefficients $X_{i,j} \in \R$, there is a quadratic form $q = \sum_{\{i,j\} \in \skel 1 \ASC} X_{i,j} x_i x_j$.
(Note that we will use the notation $\{i,j\}$ even when $i = j$, in which case, $X_{i,j} = X_{i,i} = X_{j,j}$).

We can naturally format the coefficients $X_{i,j}$ into a \textbf{partial matrix}, $X$, where the rows and columns are indexed by elements of $\VER{\ASC}$.
This is a matrix where if $\{i,j\} \not \in \skel 1 \ASC$, then the entry $X_{i,j}$ is `forgotten', since the corresponding monomial $x_ix_j$ vanishes on $\mathcal{V}(\ASC)$.
We will say that the partial matrix $X$ \textbf{represents} $q$.

\subsubsection{Partial Matrices and Sum-of-Squares}
If $q$ is sum-of-squares, $q = \ell_1^2 + \ell_2^2+\dots+\ell_k^2$ for linear forms $\ell_i$, then by extending the $\ell_i$ to all of $\PROJ^n$ by linearity, there is an extension of $q$ to a sum-of-squares polynomial on $\PROJ^n$.
That is, there exists some $\hat{q} \in \CoorR_2(\PROJ^n)$, so that
\[
    \hat{q} = \sum_{i,j \le n} \hat{X}_{i,j} x_ix_j,
\]
where, $\hat{X}_{i,j} = X_{i,j}$ for $\{i,j\} \in \skel 1 \ASC$. 

Clearly, $\hat{X}$ can be formatted into a matrix. It is well known that the quadratic form $\hat{q}$ is sum-of-squares (or equivalently, nonnegative) on $\PROJ^n$ if and only if $\hat{X}$ is PSD.

Another way of saying this is that $q$ is sum-of-squares on $\mathcal{V}(\ASC)$ if and only if the partial matrix $X$ has a \textbf{completion} $\hat{X}$ so that $\hat{X}$ is PSD.
We see that there is a clear connection between PSD-completable partial matrices and sum-of-squares quadratic forms on $\mathcal{V}(\ASC)$.

\subsubsection{Partial Matrices and Nonnegativity}
We also need a coordinate-based characterization of nonnegativity.
Say that $q \in \CoorR_2(\mathcal{V}(\ASC))$, and that $X$ is a partial matrix representing $q$.
$q$ is nonnegative if and only if for every face $F \in \ASC$, $q$ is nonnegative on the coordinate subspace $\SPAN \{e_i : i \in F\}$.
Put another way, if we let $2^F$ be the simplicial complex consisting of all subsets of $F$, we obtain a natural inclusion map $\phi_F : 2^F \rightarrow \ASC$, and $q$ is nonnegative if and only if $\phi_F^*(q)$ is nonnegative for all $F \in \ASC$.

It is not hard to see in these coordinates that $\phi_F^*(q)$ is precisely a quadratic form represented by 
\[\phi_F^*(q) = \sum_{i,j \in F} X_{i,j} x_i x_j.\]
Now, $\mathcal{V}(2^F) = \PROJ^{|F|}$, so in particular, $\phi_F^*(q)$ is nonnegative if and only if the matrix representing $\phi_F^*(q)$ is PSD.

We will denote the matrix representing $\phi_F^*(q)$ by $X|_F$, and we notice that $X|_F$ can be regarded as a submatrix of $X$.
So, we have that $q$ is nonnegative if and only if the submatrices $X|_F$ are PSD for each $F \in \ASC$.

\subsubsection{Hadamard Products}
Our first application of this coordinate-based notion is the definition of a binary operation which preserves both $\POS(\ASC)$ and $\SOS(\ASC)$, the \textbf{Hadamard product}, also known as the entry-wise product.
The Hadamard product of partial matrices $X$ and $Y$ is $X \star Y$, where we have that
\[
    (X \had Y)_{i,j} = X_{i,j}Y_{i,j},
\]
for any $\{i,j\} \in \skel 1 \ASC$.

The fact that this preserves both $\POS(\ASC)$ and $\SOS(\ASC)$ follows from the Schur product theorem (see \cite{horn2013matrix}).
This operation turns $\POS(\ASC)$ and $\SOS(\ASC)$ into a semiring under addition and Hadamard product.

\subsection{An Example}
Consider the 1-dimensional simplicial complex with 4-vertices, which we can depict as a 4-cycle:

\begin{figure}[H]
\centering
\begin{subfigure}{0.5\textwidth}
  \centering
    \begin{tikzpicture}

        \node[shape=circle,draw=black] (a1) at (2,3) {1};
        \node[shape=circle,draw=black] (a2) at (2,2) {2};
        \node[shape=circle,draw=black] (a3) at (3,2) {3};
        \node[shape=circle,draw=black] (a4) at (3,3) {4};

        \path [-](a1) edge node[left] {} (a2);
        \path [-](a2) edge node[left] {} (a3);
        \path [-](a3) edge node[left] {} (a4);
        \path [-](a4) edge node[left] {} (a1);
    \end{tikzpicture}
  \caption{The 4-cycle graph $C_4$.}
  \label{fig:sub1}
\end{subfigure}%
\begin{subfigure}{0.5\textwidth}
  \centering
\[
\begin{pmatrix}
1 & 1 & 1 & 1\\
1 & 1 & 1 & 1\\
1 & 1 & 1 & 1\\
1 & 1 & 1 & 1\\
\end{pmatrix}\rightarrow
\begin{pmatrix}
1 & 1 & ? & 1\\
1 & 1 & 1 & ?\\
? & 1 & 1 & 1\\
1 & ? & 1 & 1\\
\end{pmatrix}
\]
  \caption{A partially specified matrix corresponding to $C_4$. The diagonal entries correspond to vertices of $C_4$, and the specified off-diagonal entries correspond to edges in $\EDG{C_4}$. $?$ corresponds to an unknown entry.}
  \label{fig:sub2}
\end{subfigure}
\caption{The partially specified matrix here represents the quadratic form $\sum_i x_i^2 + 2\sum_{\{i,j\}\in E(C_4)}x_ix_j$ on the variety $\mathcal{V}(C_4)$. There are many quadratic forms on $\mathcal{V}(2^{[4]})$ which restrict to the same form on $\mathcal{V}(C_n)$, for instance $\left(\sum_i x_i\right)^2$, which is a sum-of-squares completion of this form.}
\label{fig:partial}
\end{figure}
The natural inclusion $C_4 \rightarrow 2^{[4]}$ induces a projection from the sum-of-squares cone on $2^{[4]}$ to that of $C_4$.
Because $\mathcal{V}(2^{[4]})$ is isomorphic to $\PROJ^4$, we see that a sum-of-squares quadratic on $\mathcal{V}(2^{[4]})$ can be identified with a symmetric PSD matrix using coordinates.

The projection from $\SOS(2^{[4]})$ to $\SOS(C_4)$ is given by `forgetting' the entries corresponding to nonedges of a symmetric matrix.
Since $\skel 1 {{C_4}}$ has exactly 2 fewer elements than $\skel 1 2^{[4]}$, we see that we forget 2 pairs of entries in partial matrices over $C_4$.
A partial matrix represents a sum-of-squares quadratic if and only if it can be completed to a complete PSD matrix.

On the other hand, the following matrix corresponds to an element of $\POS(\ASC)$, since all of its $2\times 2$ complete submatrices are PSD, but it is not in $\SOS(\ASC)$ because it has no PSD completion.
\begin{figure}[H]
\[
    \begin{pmatrix}
        \textbf{1} & \textbf{-1} & ? & 1\\
        \textbf{-1} & \textbf{1} & 1 & ?\\
        ? & 1 & 1 & 1\\
        1 & ? & 1 & 1
    \end{pmatrix}
\]
\caption{A partial matrix with no PSD completion. The bolded $2\times 2$ submatrix is complete, and corresponds to the face $\{1,2\} \in C_4$.}
\end{figure}

\subsection{Clique Complexes and Chordal Complexes}
If $G$ is a 1-dimensional simplicial complex (or equivalently, a graph), then there is a maximal simplicial complex  $\ASC$ so that $\skel 1 \ASC = G$, which we call the \textbf{clique complex} of $G$.

Explicitly, if $G$ is a graph, then the clique complex of $G$ is the complex whose faces are subsets of the vertices of $G$ which induce a clique in $G$.
If we think of $\skel 1 \ASC $ as being a functor from $\simp$ to the category of graphs, then the clique complex functor is a right adjoint.

Such complexes appear in many places in the matrix completion literature.
Indeed, another way of stating the classic result of Grone et al. in \cite{GRONE1984109} is that $\POS(\ASC) = \SOS(\ASC)$ if and only if $\ASC$ is the clique complex of a chordal graph.
A chordal graph is a graph with no induced cycles of size at least 4.

To emphasize the importance of chordal graphs in this work, we will make a definition.
\begin{definition}\label{def:chordal}
    A simplicial complex is said to be \textbf{chordal} if it is the clique complex of a chordal graph.
\end{definition}
\begin{thm}\label{thm:nonnegative_quads_chordal}
    (Grone et al.) If $\ASC$ is a simplicial complex, then $\POS(\ASC) = \SOS(\ASC)$ if and only if $\ASC$ is chordal.
\end{thm}

One advantage of our category theoretic framework is that we can try to leverage Theorem \ref{thm:nonnegative_quads_chordal} more extensively by writing a more general simplicial complex as the image of a chordal complex under a simplicial map, and then applying our category theoretic machinery.
To make this precise, we make a definition.
\begin{definition}\label{def:chordal_quotient}
    A simplicial map $\phi : \ASCb \rightarrow \ASC$ is said to be a \textbf{chordal quotient} if $\ASCb$ is a chordal complex, and $\phi$ is surjective in the sense that for every face $S \in \ASC$, there is some face $T \in \ASCb$ so that $\phi(T) = S$.
\end{definition}
One important note is that if $\phi$ is a surjective simplicial map, then $\phi^*$ is injective.

\subsection{Local Rank}
One last ingredient we will need to state our results is that of the local rank of a nonnegative quadratic form over a Stanley-Reisner variety.

Firstly, recall that the rank of a sum-of-squares form $q$, $\rank(q)$, is the smallest number $k$ so that there exists $k$ linear forms $\ell_1 ,\dots, \ell_k$ so that $q = \ell_1^2  +\dots+ \ell_k^2$.

\begin{definition}
    For $F \in \ASC$, we will denote the inclusion map $i_F : 2^{F} \rightarrow \ASC$. For any $q \in \POS(\ASC)$, we define the \textbf{local rank}
    \[
        \locr(q) = \max \{\rank(i^*_F(q)) : F \in \ASC \}.
    \]
\end{definition}
Equivalently, if $X$ is a partial matrix representing $q$, $\locr(q) = \max \{\rank X|_F :  F\in \ASC\}$.

One nice thing about this definition is that it agrees with the usual notion of rank for chordal complexes:
\begin{lemma}\label{lem:local_rank_cover}
    Let $q \in \POS(\ASC)$, and let $\phi$ be an arbitrary chordal quotient of $\ASC$. The \textbf{local rank} of $q$ satisfies
    \[
        \locr(q) = \rank(\phi^*(q)).
    \]
\end{lemma}
Note that the statement of this lemma makes use of the fact that for a chordal complex, any nonnegative quadratic form is sum-of-squares.
We will defer the proof of this lemma to a later section.

The \textbf{extreme local rank} of a simplicial complex $\ASC$ is the maximum local rank of any extreme ray of $\POS(\ASC)$ which is not in $\SOS(\ASC)$, or 0 if $\ASC$ is chordal.
We will denote the extreme local rank of $\ASC$ by $\elocr(\ASC)$.
Our main interest is going to be in computing the extreme local rank of some classes of simplicial complexes.

\section{Results}
\subsection{Small Quotients of Chordal Graphs}

We first present a structural result in which a complex which is `close' to being chordal has the property that $\POS(\ASC)$ is `close' to $\SOS(\ASC)$, in the sense that all of the non-sum-of-squares extreme rays of $\POS(\ASC)$ are low rank.

The \textbf{chordal deficiency} of a complex $\ASC$ is defined by
\[
    \rho(\ASC) = \min \{|\VER \ASCb| - |\VER \ASC| : \phi : \ASCb \rightarrow \ASC \text{ is a chordal quotient}\}.
\]
Like the \textbf{minimum-fill-in} of a graph (which is the smallest number of edges that need to be added to a graph to make it chordal), this is a measurement of how far a complex is from being chordal.
\begin{thm}[Theorem \ref{thm:small_chordal_quotient}]
    \label{thm:restate_chordal_quotient}
    \[
        \elocr(\ASC) \le \rho(\ASC).
    \]
\end{thm}
Notice that if $\ASC$ is chordal, then $\rho(\ASC) = 0$, so in this case, this is just a restatement of the fact that for chordal complexes, $\POS(\ASC) = \SOS(\ASC)$.

We can strengthen this result slightly when the chordal quotient has some additional structure.
\begin{thm}[Theorem \ref{thm:odd_clique}]
    \label{thm:restate_odd_clique}
    Let $\ASCb$ be a chordal complex.
    Let $S_1, S_2 \in \ASCb$ be disjoint faces, so that $|S_1| = |S_2| = k > 1$ and $k$ is odd.
    We will also suppose that there are no vertices with edges to both $S_1$ and $S_2$,
    Let $f : S_1 \rightarrow S_2$ be a bijection.
    Let $\phi : \ASCb \rightarrow \ASC$ be the quotient map obtained by identifying $x$ with $f(x)$ for each $x \in S_1$.
    We can conclude
    \[\elocr(\ASC) \le k - 1.\]
\end{thm}
Notice that the map here $\phi : \ASCb \rightarrow \ASC$ is a chordal quotient, and in particular, the previous theorem implies that $\rho(\ASC) \le k$, so this conclusion is slightly stronger.
It is perhaps surprising that this parity constraint on $k$ appears in the conditions of this theorem, but it is due to the fact that odd dimensional real vector spaces, any linear automorphism has a real eigenvector.

\subsection{Operations Preserving Local Rank}
One hallmark of a good definition is that it should be preserved by at least some natural operations.
Here, we will give two somewhat natural operations that preserve local rank.
\subsubsection{Cones}
We say that a vertex is a \textbf{cone vertex} of $\ASC$ if it is contained in every facet of $\ASC$.
Given a complex $\ASC$, the \textbf{cone} over $\ASC$ is defined as 
\[
    \hat{\ASC} = \{S \subseteq \VER{\ASC} \cup \{*\} : S \in \ASC \text{ or }S - * \in \ASC\}.
\]
Clearly, $*$ is a cone vertex in $\hat{\ASC}$.
We will see that the cone operation does not significantly impact the structure of the cone $\POS(\ASC)$.

\begin{thm}[Theorem \ref{thm:extremeRayCone}]
    Every extreme ray of $q \in \POS(\hat{\ASC})$ is either sum-of-squares, or it satisfies $q(e_*) = 0$.
\end{thm}

The set of points $C = \{q \in \POS(\hat{\ASC}) : q(e_*) = 0\}$ is a face of $\POS(\hat{\ASC})$ (this is because if $q = q_1 + q_2$, where these are all nonnegative, and $q(e_*) = 0$, then $q_1(e_*) = q_2(e_*) = 0$).
$C$ can be seen to be isomorphic to $\POS(\ASC)$, under the map $\iota^* : \POS(\hat{\ASC}) \rightarrow \POS(\ASC)$, where $\iota : \ASC \rightarrow \hat{\ASC}$ is the natural inclusion map.

In particular, $\iota^*$ preserves local rank on $C$, so that this theorem can easily be seen to have a corollary.
\begin{cor}
   \[\elocr(\ASC) = \elocr(\hat{\ASC}).\] 
\end{cor}
\subsubsection{1-Sums}
Given two complexes, $\ASC$ and $\ASCb$ on disjoint vertex sets, and vertices $a \in \ASC$ and $b \in \ASCb$, then the \textbf{1-sum} of $\ASC$ and $\ASCb$ is the complex obtained by identifying $a$ and $b$ in the disjoint union $\ASC \sqcup \ASCb$. 
We denote this 1-sum by $\ASC \oplus \ASCb$ (where we intentionally suppress the dependence on the vertices $a$ and $b$ chosen).

\begin{thm}[Theorem \ref{thm:1sum}]\label{thm:restate_1sum}
    \[\elocr(\ASC \oplus \ASCb) = \max(\elocr(\ASC), \elocr(\ASCb))\]
\end{thm}

\subsection{Locally Rank 1 Complexes}

We are particularly interested in complexes with extreme local rank at most 1 because for these, it is possible to use the structure of the extreme rays to get quantitative bounds on the distance between $\POS(\ASC)$ and $\SOS(\ASC)$.
In some senses, locally rank 1 complexes are `almost' chordal, and the extreme rays of the cones of nonnegative quadratics are `almost' squares of linear forms.

Explicitly, let 
\[
    \EXSimp_1 = \{\ASC \in \simp : \elocr(\ASC) \le 1\}.
\]

\subsubsection{Thickened Graphs}
We consider the next definition to be our main contribution, which gives us a rich class of complexes in $\EXSimp_1$ that can be obtained from `gluing together' various chordal graphs in a controlled way.

A \textbf{thickened graph} is essentially a complex in which a graph, $G$, has each of its edges replaced by an arbitrary chordal complex.
We will give a formal definition of a thickened graph in \ref{def:thick_graphs}, but provide an image describing this construction in figure \ref{fig:thickened_graph}.
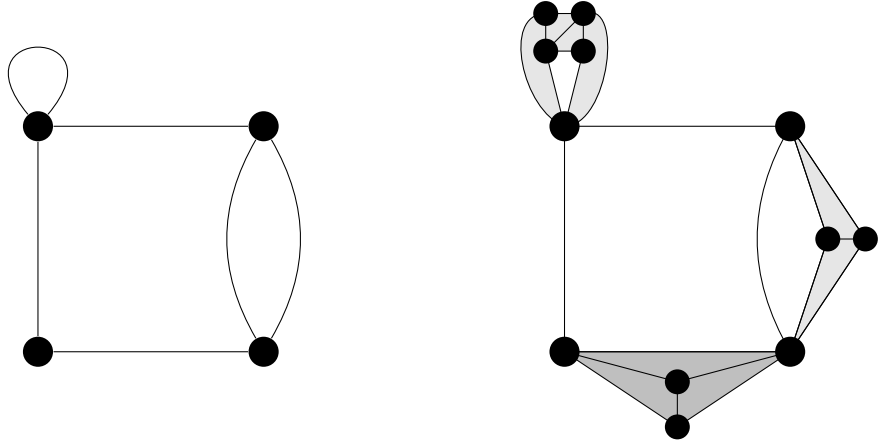
\begin{figure}[H]
    \centering
\begin{tikzpicture}[
    mainnode/.style={circle, fill=black, minimum size=4mm},
    smallnode/.style={circle, fill=black, minimum size=1mm},
    ]
    \node[mainnode] (1) at (0,0) {}; 
    \node[mainnode] (2) at (0,3) {}; 
    \node[mainnode] (3) at (3,0) {}; 
    \node[mainnode] (4) at (3,3) {}; 

    \draw [-](2) to [out=50, in=130,looseness=15] (2);
    \path [-](1) edge node[left] {} (2);
    \path [-](1) edge node[left] {} (3);
    \path [-](2) edge node[left] {} (4);
    \path [-](3) edge [bend left] node[left] {} (4);
    \path [-](3) edge [bend right] node[left] {} (4);

    \coordinate (5) at (7,0);  
    \coordinate (6) at (7,3);  
    \coordinate (7) at (10,0); 
    \coordinate (8) at (10,3); 

    \coordinate  (9) at (10.5,1.5); 
    \coordinate  (10) at (11,1.5); 

    \coordinate (11) at (8.5,-1) {}; 
    \coordinate (12) at (8.5,-0.4) {}; 

    \coordinate (13) at (6.75,4) {}; 
    \coordinate (14) at (7.25,4) {}; 
    \coordinate (16) at (6.75,4.5) {}; 
    \coordinate (15) at (7.25,4.5) {}; 

    \filldraw[draw=black, fill=gray!20] (7) -- (9) -- (8) -- (10) -- cycle;
    \filldraw[draw=black, fill=gray!50] (7) -- (5) -- (11) -- cycle;
    \filldraw[draw=black, fill=gray!20] (6) -- (13) -- (14) -- (6) to [in = 20, out = 0] (15) -- (16) to [in = 160, out = 180] cycle;

    \node[mainnode] (5n) at (7,0) {};  
    \node[mainnode] (6n) at (7,3) {};  
    \node[mainnode] (7n) at (10,0) {}; 
    \node[mainnode] (8n) at (10,3) {}; 

    \node[smallnode] (9n) at (10.5,1.5) {}; 
    \node[smallnode] (10n) at (11,1.5) {}; 
    \node[smallnode] (11n) at (8.5,-1) {}; 
    \node[smallnode] (12n) at (8.5,-0.4) {}; 

    \node[smallnode] (13n) at (6.75,4) {}; 
    \node[smallnode] (14n) at (7.25,4) {}; 
    \node[smallnode] (16n) at (6.75,4.5) {}; 
    \node[smallnode] (15n) at (7.25,4.5) {}; 

    \path [-](5) edge node[left] {} (6);
    \path [-](5) edge node[left] {} (7);
    \path [-](6) edge node[left] {} (8);
    \path [-](7) edge [bend left] node[left] {} (8);

    \path [-](7) edge node[left] {} (9);
    \path [-](7) edge node[left] {} (10);
    \path [-](8) edge node[left] {} (9);
    \path [-](8) edge node[left] {} (10);
    \path [-](9) edge node[left] {} (10);

    \path [-](11) edge node[left] {} (12);
    \path [-](7) edge node[left] {} (12);
    \path [-](5) edge node[left] {} (12);

    \path [-](13) edge node[left] {} (15);
    \path [-](14) edge node[left] {} (15);
    \path [-](13) edge node[left] {} (16);
\end{tikzpicture}
    \caption{An example of a thickened graph. To the left, is a graph, and to the right is a thickening, where some of the edges have been replaced by other chordal complexes.}%
    \label{fig:thickened_graph}
\end{figure}

The class of thickened graphs contains both the purely 1-dimensional complexes, and all chordal complexes, but also includes more interesting complexes.

\begin{thm}[Theorem \ref{thm:sufficient_rank_1}]\label{thm:restate_sufficient_rank_1}
    $\EXSimp_1$ contains all thickened graphs, and is closed under the 1-sum and cone operations.
\end{thm}

\subsubsection{Diagonal Classes of Locally Rank 1 Quadratics}
We can fully characterize the locally rank 1 forms in $\POS(\ASC)$, for any complex $\ASC$.

There is a torus action on $\mathcal{V}(\ASC)$, induced by the natural action of the group of invertible diagonal matrices acting on projective space.
That is, we can think of the torus as being the group of invertible diagonal matrices, which acts linearly on projective space, and preserves the coordinate subspaces.

This functorially induces an action on $\CoorR_2(\ASC)$ by treating an element of the torus as an automorphism $\phi : \mathcal{V}(\ASC) \rightarrow \mathcal{V}(\ASC)$.
In coordinates, we consider an $n\times n$ invertible diagonal matrix $D$ and a partial matrix $X$ representing $q \in \CoorR_2(\ASC)$, and the action is given by 
\[(D \cdot X)_{ij} = D_{ii}D_{jj}X_{i,j}. \]

If $q_1, q_2 \in \CoorR_2(\ASC)$, we say that $q_1$ and $q_2$ are \textbf{diagonally congruent} if there is a some torus automorphism on $\mathcal{V}(\ASC)$ whose induced map on $\POS(\ASC)$ sends $q_1$ to $q_2$.
It is not hard to see that local rank is preserved by this torus action.

Hence, the locally rank 1 forms in $\POS(\ASC)$ are divided into orbits of this group action.
We say that an orbit $[q]$ is of full support if $q(e_i) > 0$ for all $i\in \VER{\ASC}$.
Let $K_1$ denote the set of all orbits of locally rank 1 forms in $\POS(\ASC)$ of full support.
It turns out that the Hadamard product is well defined on $K_1$, and makes $K_1$ into a semigroup.
\begin{thm}[Theorem \ref{thm:classify_lrk1_ext}]\label{thm:restate_classify_lrk1_ext}
    The semigroup $K_1$ is isomorphic to the simplicial cohomology group $H^1(\ASC, \Z/2\Z)$.
\end{thm}

Note that if $\ASC \in \EXSimp_1$, then all extreme rays of $\POS(\ASC)$ are classified using this theorem.

The condition that $q$ is of full support is only for technical reasons; we see that for any $A \subseteq \VER{\ASC}$, $\{q \in \POS(\ASC) : \forall i \not \in A,\; q(e_i) = 0\}$ is a face of $\POS(\ASC)$ which is isomorphic to $\POS(\ASCb)$, where $\ASCb = \{F \in \ASC : F \subseteq A\}$.
Hence, to classify extxreme rays without full support, it suffices to apply the theorem to vertex induced subcomplexes of $\ASC$.

\subsection{Strongly Connected Maps}
It is not hard to see that if $\phi : \ASC \rightarrow \ASCb$ is any surjective simplicial map, then if $q \in \POS(\ASCb)$, and $\inphi(q)$ spans an extreme ray of $\POS(\ASC)$, then $q$ spans an extreme ray of $\POS(\ASCb)$.

We want to consider maps $\phi:\ASC \rightarrow \ASCb$ for which the converse also holds: if $q$ spans an extreme ray of $\POS(\ASCb)$, then $\inphi(q)$ spans an extreme ray of $\POS(\ASC)$.

\begin{definition}[Strongly Connected Map]\label{def:restate_strong_connected_map}
We say a surjective map $\phi:\ASC \rightarrow \ASCb$ is \textbf{strongly connected} if it satisfies the following two properties:
\begin{enumerate}
    \item If $a, b \in \VER \ASC$, and $\phi(a) = \phi(b) = c \in \VER{{\ASCb}}$, then there is a sequence $a = x_1 ,x_2,x_3,\dots, x_k = b \in \VER \ASC$ so that $\phi(x_i)=c$ and $\{x_i, x_{i+1}\} \in \ASC$ for each $i \in [k-1]$.
    \item If $a, b \in \EDG{\ASC}$, and $\phi(a) = \phi(b) = c \in {\stskel 1 {\ASCb}}$, then there is a sequence $a = e_1 ,e_2,e_3,\dots, e_k = b \in \EDG{\ASC}$ so that $\phi(e_i) = c$, $e_i \cup e_{i+1} \in \ASC$, and $e_i \cap e_{i+1} \neq \varnothing$ for each $i \in [k-1]$.
\end{enumerate}
\end{definition}
In essence, we want the preimage of every edge and vertex of $\ASCb$ to be 2-connected, meaning that there is a path between every pair of vertex, and a `path of  triangles' between every pair of edges
There are a number of interesting strongly connected maps.
An example which we will not descibe in depth is that if we have a clique complex $\chi(G)$ and we contract an edge $e$ of $G$ to get $G / e$, then the resulting quotient map $\phi : \chi(G) \rightarrow \chi(G / e)$ is strongly connected as long as $e$ is not contained in any induced 4-cycles.
The reason we are interested in strongly connected maps is that they give us information about the facial structure of $\POS(\ASC)$.

\begin{thm}[Theorem \ref{thm:strong_conn}]\label{thm:restate_strong_conn}
    If $\phi : \ASC \rightarrow \ASCb$ is strongly connected, then the image of $\inphi$ is a face of $\POS(\ASC)$.
\end{thm}

In particular, this result implies that extreme rays of $\POS(\ASCb)$ pull back to extreme rays of $\POS(\ASC)$.

\begin{cor}\label{cor:strong_extreme_rays}
    If $\phi : \ASC \rightarrow \ASCb$ is strongly connected, then for any $q \in \POS(\ASCb)$ which spans an extreme ray, $\inphi(q)$ spans an extreme ray of $\POS(\ASC)$.
\end{cor}

It is clear that if $\phi$ is surjective, then $\phi$ preserves local rank, so that as a corollary,
\begin{cor}[Extreme Local Rank and Strongly Connected Maps]\label{cor:local_rank_sconn}
    If $\phi : \ASC \rightarrow \Gamma$ is strongly connected, then $\elocr(\ASC) \ge \elocr(\Gamma)$.    
\end{cor}

Our next collection of results gives further examples coming from algebraic topology.

\subsection{Maps to Spheres and Combinatorial Manifolds}
We would like to apply theorem \ref{thm:strong_conn} to obtain lower bounds on the extreme local rank for some complexes.
To do so, we first wish to consider some complexes with relatively large extreme local rank.

One such class of complexes is the set of simplicial spheres, defined as 
\[
    S_d = \{S \subseteq [d+2] : |S| \le d+1\}.
\]

It was shown in \cite[Theorem 4.4]{gouveia2021sums} what the extreme local rank of simplicial spheres are.
\begin{lemma}[Extreme Local Rank of Spheres]\label{lmma:elr_spheres}
    \[
        \elocr(S_d) = d.
    \]
\end{lemma}
Thus, if we wish to show that $\ASC$ has extreme local rank at least $d$, it suffices to obtain a strongly connected map from $\ASC$ to $S_d$.
\begin{thm}[Maps to Spheres]\label{thm:restate_map_to_sphere}
    Suppose that $\ASC$ is a simplicial complex, and $F$ is a facet of $\ASC$ of size $d+1 > 1$ with the following properties.
    (Use $F^c$ to denote $\VER{\ASC}-F$.)
    \begin{enumerate}
        \item For any $a, b \in F^c$, there is a sequence $a_1 ,\dots, a_k \in F^c$ so that $a_1 = a$, $a_k = b$ and $\{a_i, a_{i+1}\} \in \ASC$ for each $i$.
        \item For each $c \in F$, and for any $a, b \in F^c$ so that $\{a,c\}, \{b,c\} \in \ASC$, there is a sequence $a_1 ,\dots, a_k \in F^c$ so that $a_1 = a$, $a_n = b$ and $\{a_i, a_{i+1}, c\} \in \ASC$ for each $i \in [n]$.
        \item For each $c \in F$, there is some $w \in F^c$ so that $F - c + w \in \ASC$.
    \end{enumerate}

    Then there is strongly connected map from $\ASC$ to $S_d$.
\end{thm}

A particularly illustrative (though not exhaustive) collection of examples come from \textbf{combinatorial manifolds}.
To define these, we will need to define the \textbf{link} of a simplicial complex $\ASC$ at a vertex $i \in E$ to be 
\[
    \lnk_i(\ASC) = \{F \in \ASC : i \not \in F,\; F \cup \{i\} \in \ASC\}.
\]
$\lnk_i(\ASC)$ is clearly a simplicial complex.
\begin{definition}[Combinatorial Manifold]\
A simplicial complex $\ASC$ is said to be a \textbf{combinatorial manifold} of dimension $d$ if for each vertex $i \in \VER \ASC$, the link $\lnk_i(\ASC)$ has a topological realization which is homemomorphic to a sphere of dimension $d-1$, and $|\ASC|$ is also a topological manifold.

\end{definition}
We found this definition in  \cite{penna1978geometry}.
Note that some definitions of combinatorial manifolds require that the link of each vertex be simplicially isomorphic to $S_d$, which is a stronger condition.

\begin{cor}\label{cor:combin_manifolds}
    If $\ASC$ is a connected combinatorial manifold of dimension $d$, then $\ASC$ has a strongly connected map to $S_d$.
    In particular,
    \[\elocr(\ASC) \ge d.\]
\end{cor}

\begin{remark}
    It is noteworthy that all of the complexes which we know have extreme local rank 1 have the property that they, and all of their vertex induced subcomplexes, are homotopy equivalent to purely 1-dimensional complexes (though there are examples for which the converse does not hold).
    On the other hand, these results show that a complex with a vertex induced subcomplex which is a combinatorial manifold of dimension $d > 1$ does not have extreme local rank 1.
    It is natural to ask then whether the following generalization of these facts hold:
    \begin{conj}
        If $\ASC$ is a complex with extreme local rank 1, then $H^i(\ASC, \mathbb{Z}) = 0$ for $i > 1$. 
    \end{conj}
\end{remark}

The remainder of this paper is devoted to proofs of the results listed above.
\section{Technical Results}
In order to give our proofs, we will need a few more notions that will allow us directly to decompose nonnegative quadratics into sums.
The main technical tool we will need is \ref{lem:decompositions}, which will enable us to find decompositions of quadratic forms with high local rank into quadratic forms with smaller local rank using dimension counting arguments.
\subsection{Coordinatizing Maps Induced by Simplicial Maps}
It will be useful to be able to write down induced maps on quadratic forms in terms of partial matrices.
The following lemma is not hard to show.
\begin{lemma}\label{lem:ind_coor}
    If $\phi : \ASC \rightarrow \ASCb$ is a simplicial map, and $q \in \POS(\ASCb)$ is represented by a partial matrix $X$, then $\inphi(q)$ is represented by a partial matrix $Y$, where for each $\{i,j\} \in \skel 1 \ASC$,
    \[
        Y_{i,j} = X_{\phi(i), \phi(j)}.
    \]
\end{lemma}
In order to decompose elements in the image of some induced map $\inphi$, it will also be useful to have some equivalent conditions for the image.
\begin{lemma}\label{lem:image_of_inphi}
    Let $\phi : \ASC \rightarrow \ASCb$ be a surjective simplicial map.
    The image of $\phi^* : \POS(\ASCb) \rightarrow \POS(\ASC)$ are those nonnegative forms which are represented by partial matrices $X$ with the following property: for any $\{i,j\}, \{\ell, k\} \in \skel 1 \ASCb $ with $\phi(\{i,j\}) = \phi(\{\ell, k\})$,
    \[
        X_{i, j} = X_{\ell, k}.
    \]
\end{lemma}
\begin{proof}
    Suppose that $q \in \POS(\ASC)$ has the desired property; we wish to show that $q$ is the image of some $q' \in \ASCb$.
    Let $Y$ be a partial matrix representing a quadratic form $r \in \CoorR_2(\mathcal{V}(\ASCb))$ defined so that for any $\{i,j\} \in \EDG{\ASCb}$,
    \[
        Y_{i,j} = X_{a, b},
    \]
    for some (any) $\{a,b\}$ so that $\phi(\{a,b\}) = \{i,j\}$.

    By lemma \ref{lem:ind_coor}, we have that the partial matrix $X'$ representing $\phi^*(r)$ satisfies
    \[
        X'_{a,b} = Y_{\phi(a),\phi(b)} = X_{a.b}.
    \]

    Now, we must show that $r$ is nonnegative, which means that the submatrices of $Y$ corresponding to faces of $\ASCb$ must be PSD.
    Let $F \in \ASCb$ be a face, and let $G \in \ASC$ so that $\phi(G) = F$.
    By passing to a subset of $G$, we can assume that $\phi$ restricts to a bijection of $F$ and $G$.

    Let $Y|_F$ denote the submatrix of $Y$ defined by $F$,

    We see then that under this bijection, $Y|_F$ and $X|_G$ are equal, i.e. for any $i,j \in G$,
    \[Y_{\phi(i),\phi(j)} = X_{i, j}.\]
    Because $X|_G$ is PSD, we have that $Y|_F$ is also PSD.
    Hence, $Y$ represents some quadratic form $\POS(\ASCb)$ whose image is $q$, as desired.
\end{proof}

\subsection{Positive Semidefinite Matrices as a Moduli Space}\label{subsec:psd_matrices_moduli}
Positive semidefinite matrices are of fundamental interest in a number of fields.
We will review some facts about positive semidefinite matrices and their characterization as Gram matrices.
\begin{definition}[Positive Semidefinite]\label{def:psd}
    A matrix $M \in \Sym_n$ is said to be positive semidefinite (PSD) if it satisfies any of the following (equivalent) conditions.
    \begin{enumerate}
        \item All eigenvalues of $M$ are nonnegative.
        \item The quadratic form $v^{\intercal}Mv$ is nonnegative.
        \item The quadratic form $v^{\intercal}Mv$ is sums of squares.
        \item There exists a Hilbert space $H$ and vectors $v_1 ,\dots, v_n$ so that in coordinates, $M_{ij} = \langle v_i, v_j \rangle$.
    \end{enumerate}
\end{definition}

\begin{definition}[Gram Matrix]\label{def:gram_matrix}
    If $H$ is a Hilbert space, and $v_1 ,\dots, v_n \in H$, then the Gram matrix $v_1 ,\dots, v_n$ is a matrix $G(v_1 ,\dots, v_n)$ where
    \[
        G(v_1 ,\dots, v_n)_{ij} = \langle v_i, v_j\rangle.
    \]
\end{definition}
In some cases, we will use the vertices of a simplicial complex as indices in a matrix, and we will want to use those same vertices to index the vectors.
In this case, we might use notation like $G(v_i : i \in \VER{\ASC})$ to indicate this.

A linear map between Hilbert spaces $L : H_1\rightarrow H_2$ is said to be \textbf{orthogonal} if $\langle L(v), L(w)\rangle_{H_2} = \langle v, w\rangle_{H_1}$.
The Gram matrix of a collection of vectors in a Hilbert space determines a collection of vectors up to orthogonal transformation, (see for example \cite[Theorem 7.3.11]{horn2013matrix}).
\begin{thm}[Gram Matrices Determine Vectors up to Orthogonal Transformation]\label{thm:gram_matrices_moduli}
    Suppose  $v_1 ,\dots, v_n \in H_1$, $w_1 ,\dots, w_n \in H_2$, and assume that $v_1,\dots, v_n$ span $H_1$. Then, $G(v_1 ,\dots, v_n) = G(w_1 ,\dots, w_n)$ if and only if  there is an orthogonal map $L : H_1 \rightarrow H_2$ so that for each $i \in [n]$,
    \[
        L(v_i) = w_i.
    \]
\end{thm}

Because a Gram matrix uniquely specifies a collection of vectors up to orthogonal transformation, many basic properties about a PSD matrix that can be read off of a Gram matrix representation.
\begin{lemma}[Rank]\label{lem:gram_matrices_rank}
    If $M = G(v_1 ,\dots, v_n)$, then the rank of $M$ is equal to the dimension of $\SPAN(v_1 ,\dots, v_n)$.
\end{lemma}

Moreover, basic operations on the level of Gram matrices correspond to basic operations on the level of Hilbert spaces.
We will list some of these operations without proof, for the sake of completeness, but the proofs follow easily from the definitions.
The fact which is key to these operations is that they operate entrywise on a matrix, so that they can also be applied to partial matrices, and by extension, nonnegative quadratics over Stanley-Reisner varieties.
\begin{lemma}
    If $v_1 ,\dots, v_n \in H_1$, $w_1 ,\dots, w_n \in H_2$, and $v_i \oplus w_i \in H_1 \oplus H_2$ denotes the direct sum of $v_i$ and $w_i$, then
    \[
        G(v_1 ,\dots, v_n) +  
        G(w_1 ,\dots, w_n) = 
        G(v_1 \oplus w_1,\dots, w_n \oplus v_n).
    \]\label{lem:gram_matrices_sum}

\end{lemma}
\begin{lemma}\label{lem:hadamard}
    If $v_1 ,\dots, v_n \in H_1$ and $w_1 ,\dots, w_n \in H_2$, so that $v_i \otimes w_i \in H_1 \otimes H_2$ in the tensor product of $H_1$ and $H_2$, then
    \[
        G(v_1 ,\dots, v_n) \star G(w_1 ,\dots, w_n) = G(v_1\otimes w_1, \dots, v_n\otimes w_n).
    \]
    (recall that $\had$ denotes the \textbf{Hadamard product})

\end{lemma}
\begin{lemma}\label{lem:schur_proj}
    If $v_1 ,\dots, v_n \in H$, and $\pi : H \rightarrow H$ denotes the projection of $H$ onto the orthogonal complement of $v_1$, then
    \[
        G(\pi(v_2) ,\dots, \pi(v_n)) = G(v_1 ,\dots, v_n) / (1,1),
    \]
    where $A / (1,1)$ denotes the \textbf{Schur complement} of $A$ with respect to the $1\times 1$ submatrix $A|_{\{1\}}$, given by
    \[
        A / (1,1) = A|_{\{2 ,\dots, n\}} - A_{11}^{-1} A_{1\cdot}A_{\cdot 1}.
    \]
    ($A_{\cdot i}$denotes the $i^{th}$ column of $A$ and  $A_{i \cdot}$ denotes the $i^{th}$ row of $A$.)
\end{lemma}
\subsection{Decompositions for $\POS(\ASC)$}
In this paper, we will want to argue that elements of $\POS(\ASC)$ with large local rank are not extreme points.
For that, we will need ways of decomposing $q \in \POS(\ASC)$ into a sum of $q_1, q_2 \in \POS(\ASC)$, where $q_1$ and $q_2$ are lower rank than $q$.
To do this, we will characterize the ways of decomposing PSD matrices, and more generally, nonnegative quadratics over Stanley-Reisner varieties, into sums of PSD matrices.

The lemmas \ref{lem:ind_coor}, \ref{lem:gram_matrices_sum} and \ref{thm:gram_matrices_moduli} can be combined together to understand the possible decompositions of a PSD matrix $X$ as the sum of two other PSD matrices in a more structured way.
Specifically, let $v_1 ,\dots, v_n \in H$ for some Hilbert space $H$.
Further let $T : H \rightarrow R$ be an orthogonal map of Hilbert spaces, and let $L \subseteq R$ be some linear subspace of $R$.
Let $\pi_L : R \rightarrow L$ denote the projection of $R$ onto $L$.
We say that the \textbf{$(T, L)$-decomposition} of $G(v_1 ,\dots, v_n)$ is the decomposition
\[
    G(v_1 ,\dots, v_n) = G(\pi_L(Tv_1) ,\dots, \pi_L(Tv_n))+ G(\pi_{L^{\intercal}}(Tv_1) ,\dots, \pi_{L^{\intercal}}(Tv_n)).\tag{*}
\]
We will usually be interested in the case when $T$ is the identity transformation.
\begin{lemma}\label{lem:TLdecomp}
    The equation (*) in the definition of a $(T,L)$ decomposition is valid.
    Conversely, whenever $X = G(v_1 ,\dots, v_n)$ and $G(v_1 ,\dots, v_n) = G(w_1 ,\dots, w_n) + G(u_1 ,\dots, u_n)$, then this has the form of a $(T,L)$-decomposition.
\end{lemma}
\begin{proof}
    Firstly, we show that the definition of a $(T,L)$-decomposition is in fact valid. Because $T : H \rightarrow R$ is orthogonal, we have the equality of Gram matrices,
    \[
        G(v_1 ,\dots, v_n) = G(Tv_1, \dots, Tv_n).
    \]
    Notice that the decomposition $H = L \oplus L^{\intercal}$ implies that the map $\pi_L \oplus \pi_{L^{\intercal}}$ is orthogonal as well, so that
    \[
        G(v_1 ,\dots, v_n) = G(\pi_L(Tv_1)\oplus G(\pi_{L^{\intercal}}(Tv_1)), \dots, \pi_L(Tv_n) \oplus\pi_{L^{\intercal}}(Tv_n)).
    \]
    Finally, we apply lemma \ref{lem:gram_matrices_sum} to obtain that
    \[
        G(v_1 ,\dots, v_n) = G(\pi_L(Tv_1) ,\dots, \pi_L(Tv_n))+ G(\pi_{L^{\intercal}}(Tv_1) ,\dots, \pi_{L^{\intercal}}(Tv_n)).
    \]

    On the other hand, we can apply \ref{lem:gram_matrices_sum} in the following expression:
    \begin{align*}
        G(v_1 ,\dots, v_n) &= G(w_1 ,\dots, w_n) + G(u_1 ,\dots, u_n)\\
                       &= G(w_1 \oplus u_1,\dots, w_n\oplus u_n).
    \end{align*}
    Here, $w_1 ,\dots, w_n \in W$, and $u_1 ,\dots, u_n \in U$ for Hilbert spaces $U$ and $W$.
    Theorem \ref{thm:gram_matrices_moduli} then implies that there is an orthogonal transformation $T : H \rightarrow W\oplus U$ sending $v_i$ to $w_i\oplus u_i$.
    Setting $L = W$ in this decomposition, we obtain our desired $(T, L)$-decomposition.
\end{proof}
Now, we can generalize the notion of a $(T,L)$-decomposition for chordal covers.

Let $\phi : \ASC \rightarrow \ASCb$ be a chordal cover, and let $q \in \POS(\ASCb)$.
Let $X$ be a partial matrix representing $\inphi(q)$, and let $\hat{X}$ be a PSD matrix completing $X$ (which exists because $\ASC$ is a chordal complex).
Suppose that $\hat{X} = G(v_1 ,\dots, v_n)$ for $v_1 ,\dots, v_n \in H$.
A pair $(T,L)$ is said to be \textbf{compatible} with $\phi$ and $v_1 ,\dots, v_n \in H$ if for all pairs $\{i,j\} \in \ASC$ and $\{k,\ell\} \in \ASC$ so that $\phi(\{i,j\}) = \phi(\{k,\ell\})$, we have

\[ 
    \langle \pi_L(T(v_i)), \pi_L(T(v_j))\rangle = \langle \pi_L(T(v_k)), \pi_L(T(v_{\ell}))\rangle.
\]

Let $\hat{X} = \hat{Y} + \hat{Z}$ be $(T,L)$-decomposition, for some $(T,L)$ compatible with $\phi$ and the $v_i$.
 We say that it \textbf{induces} a decomposition $q = q_1 + q_2$ if $X$ represents $\phi^*(q)$, $Y$ represents $\phi^*(q_1)$, and $Z$ represents $\phi^*(q_2)$.
\begin{lemma} \label{lem:decompositions}
    Let $\phi : \ASCb \rightarrow \ASC$ be a chordal cover, and let $q \in \POS(\ASC)$.
    Let $X$ be a partial matrix representing $\phi^*(q)$.

    Let $\hat{X} = G(v_1, \dots, v_n)$ be any PSD completion of $X$. Then, any $(T,L)$-decomposition of $\hat{X}$ compatible with $\phi$ and $v_1, \dots, v_n$ induces a decomposition $q = q_1 + q_2$ for $q, q_1, q_2 \in \POS(\ASC)$.

    Conversely, for any decomposition $q = q_1 + q_2$, there exists some PSD completion of $X$, $\hat{X}=  G(v_1, \dots, v_n)$, and a $(T,L)$-decomposition of $\hat{X}$ compatible with $\phi$ and $v_1, \dots, v_n$ inducing the decomposition.
\end{lemma}
\begin{proof}
    Firstly, if $\hat{X} = G(v_1 ,\dots, v_n)$, and $(T,L)$ is compatible with $\phi$ and $v_1 ,\dots, v_n$, then let 
    \[
        \hat{A} = G(\pi_L(T(v_1)) ,\dots, \pi_L(T(v_n))).
    \]

    For any $i,j \in \EDG{\ASCb}$,
    \[
        A_{i,j} = \langle \pi_L(T(v_i)), \pi_L(T(v_j))) \rangle.
    \]

    If $\phi(\{i,j\}) = \phi(\{k,\ell\})$, then because $(T,L)$ is compatible with $\phi$ and $v_1 ,\dots, v_n$, we obtain
    \[
        A_{i,j} = \langle \pi_L(T(v_i)), \pi_L(T(v_j))) \rangle = \langle \pi_L(T(v_k)), \pi_L(T(v_{\ell}))) \rangle = A_{k,\ell}.
    \]

    We conclude from \ref{lem:image_of_inphi} that $A$ is in the image of $\phi^*$. Similarly, we have that $B$ is in the image of $\inphi$.
    We then know that $A$ represents $\phi^*(q_1)$, and $B$ represents $\phi^*(q_2)$, for $q_1,q_2\in \POS(\ASC)$, so that $\phi^*(q) = \phi^*(q_1+q_2)$.
    Because $\phi$ is surjective, $\phi^*$ is injective, and so we have that $q = q_1+q_2$.

    For the other direction, let $q=q_1+q_2$.
    Let $A$ and $B$ be partial matrices representing $\inphi(q_1)$ and $\inphi(q_2)$, and let $\hat{A}$ and $\hat{B}$ be corresponding PSD completions of $A$ and $B$.
    Letting $X = A+B$, and $\hat{X} = \hat{A}+\hat{B}$, we see that $\hat{X}$ is in fact a PSD completion of $X$.
    Let $\hat{X} = G(v_1 ,\dots, v_n)$.
    $\hat{X}=\hat{A}+\hat{B}$ is a $(T,L)$-decomposition of $X$ from lemma \ref{lem:TLdecomp}, so that
    \[
        \hat{A} = G(\pi_L(T(v_1)) ,\dots, \pi_L(T(v_n))).
    \]

    Moreover, from lemma \ref{lem:image_of_inphi}, we obtain that because $A$ represents a form in the image of $\inphi$, whenever $\phi(\{i,j\}) = \phi(\{k,\ell\})$,
    \[
        A_{i,j} = A_{k,\ell}.
    \]
    From this, we obtain that 
    \[
        \langle \pi_L(T(v_i)), \pi_L(T(v_j)) \rangle = A_{i,j} = A_{k,\ell} = 
        \langle \pi_L(T(v_k)), \pi_L(T(v_{\ell})).
    \]

    Hence, this $(T,L)$ decomposition is compatible with $\phi$ and $v_1 ,\dots, v_n$, as desired.
\end{proof}

\subsection{Chordal Covers and Local Rank}\label{subsec:local_rank}
We give a proof of lemma \ref{lem:local_rank_cover}, which will be made easier with the following lemma.
\begin{lemma}
    Let $\ASC$ be a chordal complex, and $q \in \POS(\ASC)$, then $\locr(q) = \rank q$.
\end{lemma}
\begin{proof}
    Let $\ASC = \chi(C)$, for a chordal graph $C$.
    Let $q \in \POS(\ASC)$.

    To see that $\rank q \ge \locr(q)$, notice that if $q = \ell_1^2 + \ell_2^2 +\dots+\ell_r^2$, and $\iota_F : 2^F \rightarrow \ASC$ is the inclusion of a face into $\ASC$, then $\iota_F^*(q) = \iota_F^*(\ell_1)^2 + \iota_F^*(\ell_2)^2 +\dots+\iota_F^*(\ell_r)^2$, so that the rank of $\iota_F^*(q)$ is at most $r$.

    To see that $\rank q \le \locr(q)$, we must show the following: given a partial matrix $X$ so that for each face $F \in \ASC$, $X|_F$ is PSD of rank at most $k$, there must be some PSD completion of $X$ with rank at most $k$.
    Equivalently, we need to show that there is some $k$-dimensional Hilbert space $H$, and some vectors $v_1, \dots, v_n$ so that for each $\{i,j\} \in E(C)$,
    \[X_{i,j} = \langle v_i, v_j \rangle\]

    To find these $v_i$, we will use induction, as well as a characterization of chordal graphs as being graphs with a \textbf{simplicial vertex}.
    That is, $C$ is a chordal graph if and only if either $C$ is a single vertex graph, or there is a vertex $x \in C$ so that the neighborhood of $x$ in $C$ induces a clique, and the subgraph obtained from removing $x$, $C - x$ is also chordal.

    Now, as a base case, if $C$ is a single vertex graph, then it is clear that the local rank of any positive $q$ is 1, and the rank of $q$ is also 1.

    On the other hand, if $C$ is not a single vertex graph, then suppose that the vertices of $C$ are enumerated $V(C) = \{1,\dots, n\}$, and suppose that the vertex $1$ is a simplicial vertex, whose neighborhood is $\{2,3\dots,\ell\}$.
    Consider $C'$, the graph obtained by deleting the first vertex of $C$.
    $C'$ is chordal, so by induction, there is a $k$-dimensional Hilbert space $H$ and vectors $v_2, \dots v_n$ so that for all $i,j \in E(C')$,
    \[
        X_{i,j} = \langle v_i, v_j \rangle.
    \]

    So, it remains to find $v_1$ so that for $j \in \{2,\dots, \ell\}$,
    \[
        X_{1,j} = \langle v_1, v_j \rangle.
    \]
    To do this, notice that the set of vertices $\{1, \dots, \ell\} \subseteq V(C)$ induces a clique of $C$, and hence a face of $\ASC$.
    $X|_{\{1, \dots, \ell\}}$ is PSD and rank at most $k$ by definition of the local rank, so there are vectors $w_1, \dots, w_{\ell} \in H'$ so that $H'$ is $k$-dimensional, and for $i,j \in [\ell]$,
    \[
        X_{i,j} = \langle w_i, w_j \rangle.
    \]
    Now, notice that for $i, j \in \{2, \dots, \ell\}$,
    \[
        \langle w_i, w_j \rangle = X_{i,j} = \langle v_i, v_j \rangle.
    \]
    Hence, there is an orthogonal transformation $T:\SPAN \{w_2, \dots, w_{\ell}\} \rightarrow H$ so that $T(w_i) = v_i$ for $i \in \{2, \dots, \ell\}$.
    If $\SPAN \{w_2, \dots, w_{\ell}\} \neq H'$, then we can extend the map $T$ to an orthogonal map from $H'$ to $H$, since these are both $k$-dimensional Hilbert spaces.

    So, let us take $v_1 = T(w_1)$, which gives us that for each $i \in \{1, \dots, \ell\}$,
    \[
        X_{1,j} = \langle w_1, w_j \rangle = \langle T(w_1), T(w_j) \rangle = \langle v_1, v_j\rangle.
    \]
    This gives us the desired Gram matrix representation for $X$, showing that $X$ has a rank-$k$ PSD completion.
\end{proof}

We now give a proof of lemma \ref{lem:local_rank_cover}.
\begin{proof} (of lemma \ref{lem:local_rank_cover})
    Let $\phi : \ASCb \rightarrow \ASC$ be a chordal cover, and let $q \in \POS(\ASC)$.
    We want to show that $\rank \phi^*(q) = \locr(q)$, which by the previous lemma is equivalent to saying that $\locr(\phi^*(q)) = \locr(q)$.
    
    To show that $\locr(\phi^*(q)) \ge \locr(q)$, let $F$ be a face of $\ASC$, with the inclusion of $\iota_F : 2^{F} \rightarrow \ASC$, so that $\locr(q) = \rank \iota_F^*(q)$.
    Because $\phi$ is surjective, there is some $T \in \ASCb$ so that $\phi(T) = F$, and by passing to a subset of $T$, we can assume that $\phi$ is a bijection between $T$ and $F$.
    This implies that $\iota_T^*(\phi^*(q))$ gets sent to $\iota_F^*(q)$ under the induced isomorphism of $\POS(2^T)$ and $\POS(2^F)$.
    In particular, $\rank \iota_T^*(\phi^*(q)) = \rank \iota_F^*(q)$, which implies that $\locr(\phi^*(q)) \ge \locr(q)$.

    On the other hand, to show that $\locr(\phi^*(q)) \le \locr(q)$, let $T \in \ASCb$, and let $F = \phi(T) \in \ASC$. We want to show that $\rank \iota_T^*(\phi^*(q)) = \rank \iota_F^*(q)$.
    To see that this is the case, let $\phi^*_F(q)$ be represented by the matrix $G(v_i : i \in \VER{\ASC})$, where $v_1, \dots, v_n \in H$ and $H$ has dimension $\rank \iota_F^*(q)$. 
    Then, $\iota_T^*(\phi^*(q))$ is represented by $G(v_{\phi(i)} : i \in \VER{\ASCb})$, which has rank at most $\rank \iota_F^*(q)$, since the $v_i$ still lie in a Hilbert space of dimension at most $k$.

    We see then that $\locr(\phi^*(q)) = \locr(q)$, as desired.
\end{proof}
\section{Small Chordal Quotients}
The next theorem, concerning complexes which are quotients of chordal complexes, can be considered an application of lemma \ref{lem:decompositions}.

\begin{thm}[Small Chordal Quotients]
    \label{thm:small_chordal_quotient}
    \[
        \elocr(\ASC) \le \rho(\ASC).
    \]
\end{thm}
\begin{proof}
    The result is clear if $\ASC$ is chordal, since $\rho(\ASC) = 0$ in that case, and $\elocr(\ASC)$ is defined to be 0.

    Now, suppose that $\ASC$ is not chordal, and let $\phi : \ASCb \rightarrow \ASC$ be a chordal cover so that $|\VER{\ASCb}| - |\VER{\ASC}| = \rho(\ASC)$.

    Suppose for a contradiction that $q$ spans a non-sum-of-squares extreme ray of $\POS(\ASC)$, and
     \[
         \locr(q) > |\VER{\ASC}| - |\VER{\ASCb}|.
    \]

    Let $X$ be a partial matrix which represents $\inphi(q)$.
    Because $\ASCb$ is chordal, $X$ is PSD-completable.

    By lemma \ref{lem:local_rank_cover}, there is PSD completion of $X$, $\hat{X}$, so that the rank of $\bar{X}$ is exactly $\locr(q)$.

    Because $\hat{X}$ is PSD, $\hat{X}$ is a Gram matrix, so that $\bar{X} = G(v_i : i \in \VER{\ASCb})$, where each $v_i \in H$, for some Hilbert space $H$.
    Without loss of generality, we assume that $\{v_i\}$ spans $H$, so that by lemma \ref{lem:gram_matrices_rank}, $H$ is exactly $\locr(q)$ dimensional.

    We now wish to find a $(T,L)$-decomposition of $\hat{X}$, which is compatible with $\phi$ and the $v_i$ so that $T$ is the identity map on $H$, and $L$ is 1-dimensional.
    Equivalently, we wish to find some 1-dimensional subspace $L \subseteq H$ so that for every $\{i, j\}, \{k, \ell\} \in \skel 1 \ASCb$ with $\phi(\{i,j\}) = \phi(\{k, \ell\})$,
    \[
        \langle \pi_L(v_i), \pi_L(v_j) \rangle =
        \langle \pi_L(v_k), \pi_L(v_\ell) \rangle.
    \]

    Suppose now that $L$ is the one dimensional subspace spannned by the unit vector $\unit$.
   We then have that for any $v \in H$, $\pi_L(\unit) = \langle \unit, v\rangle \unit$.
    Applying this formula for $\pi_L$, the previous equation reduces to:
    \[
        \langle \unit, v_i\rangle \langle \unit, v_j \rangle =
        \langle \unit, v_k\rangle \langle \unit, v_{\ell} \rangle
    \]

    Note now that it suffices that for each $a, b$ so that $\phi(a) = \phi(b)$,
    \[
        \langle \unit, v_a\rangle =
        \langle \unit, v_b\rangle
    \]
    This is a system of homogeneous linear equations, so to show that the desired $\unit$ exists, we just need to show that these equations span a space of dimension at most $\locr(q) - 1$.

    For each $a \in \VER{\ASC}$, choose some $x_a \in \VER{\ASCb}$ so that $\phi(x_a) = a$.
    The equations are then equivalent to the statement that for each $i \in \phi^{-1}(a) - x_a$,
    \[
        \langle \unit, v_i \rangle = 
        \langle \unit, v_{x_a} \rangle
    \]

    For each $a \in \VER{\ASC}$, this is a system of $|\phi^{-1}(a)|-1$ linear equations.

    The total number of equations is $\sum_{a \in \VER{\ASC}} (|\phi^{-1}(a)| - 1) = |\VER{\ASCb}| - |\VER{\ASC}| < \locr(q)$.
    We see then that the number of equations is less than the dimension of $H$, so there is a nonzero solution $\unit$ to this system.
    
    We conclude that there are some $q_1$ and $q_2 \in \POS(\ASC)$, so that $q = q_1 + q_2$, and $\phi^*(q_1)$ is represented by $G(\pi_L(v_1), \dots, \pi_L(v_n)))$.
    Because $L$ is 1-dimensional, the local rank of $q_1$ must then be 1, or else $q_1 = 0$.

    If $q_1 = 0$, then $\langle \omega, v_i \rangle = 0$ for all $i \in [n]$, but this would imply that $\omega = 0$, since the $v_i$ span $H$, a contradiction.
    Therefore, $q_1$ has local rank 1, and is linearly independent from $q$, which has local rank strictly greater than 1.
    This shows that $q$ is not an extreme ray.
\end{proof}

We can obtain a slightly stronger result when $\phi$ `glues together' two odd cliques of a chordal complex.
\begin{thm} \label{thm:odd_clique}
    Let $\ASCb$ be a chordal complex.
    Let $S_1, S_2 \in \ASCb$ be disjoint faces, so that $|S_1| = |S_2| = k > 1$ and $k$ is odd.
    We will also suppose that there are no vertices with edges to both $S_1$ and $S_2$,
    Let $f : S_1 \rightarrow S_2$ be a bijection.
    Let $\phi : \ASCb \rightarrow \ASC$ be the quotient map obtained by identifying $x$ with $f(x)$ for each $x \in S_1$.
    We can conclude
    \[\elocr(\ASC) \le k - 1.\]
\end{thm}
\begin{proof}
    Using theorem \ref{thm:small_chordal_quotient}, we see that every extreme ray of $\POS(\ASC)$ has local rank at most $k$. 
    Let $q$ be such an extreme ray, and suppose that $q$ has local rank exactly $k$.

    Let $G(v_i : i \in \VER{\ASCb})$ be a PSD completion of a partial matrix representing $\inphi(q)$, where the $v_i$ span $H$, a $k$-dimensional Hilbert space.

    We have from the definition of $\inphi$ that for all $i. j \in S_1$,
    \[
        \langle v_i, v_j \rangle = 
        \langle v_{f(i)}, v_{f(j)} \rangle
    \]
    This implies that $G(v_i : i \in S_1) = G(v_j : j \in S_2)$, so lemma \ref{thm:gram_matrices_moduli} implies that there is an orthogonal map $T : H \rightarrow H$ so that $T v_i = v_{f(i)}$ for all $i$.

    Because $H$ is an odd dimensional vector space, the linear map $T^{\intercal}$ has a real eigenvector, $\unit$, so that $\|\unit\| = 1$.
    Because $T$ is orthogonal, this eigenvector has eigenvalue $\lambda \in \{-1, 1\}$.

    We wish to argue that $L = \SPAN \unit$ satisfies the $(T,L)$ compatibility conditions for $T$ being the identity on $H$ to apply lemma \ref{lem:decompositions}.

    Because of our assumption that there are no vertices with edges to both $S_1$ and $S_2$, it suffices to check the following compatibility condition for $i, j \in S_1$,
    \begin{align*} 
        \langle \unit, v_i\rangle \langle \unit, v_j \rangle &= \langle \unit, v_{f(i)}\rangle \langle \unit, v_{f(j)} \rangle\\
             &= \langle \unit, Tv_i\rangle \langle \unit, Tv_{j} \rangle\\
             &= \langle T^{\intercal}\unit, v_i\rangle \langle T^{\intercal}\unit, v_{j} \rangle\\
             &= \lambda^2\langle \unit, v_i\rangle \langle \unit, v_{j} \rangle\\
             &= \langle \unit, v_i\rangle \langle \unit, v_{j} \rangle.
    \end{align*}

    Furthermore, because we assume that $k > 1$ and $G(\pi_L(v_i) : i \in \VER{\ASCb})$ has rank 1,  this $(T,L)$-decomposition produces linearly independent summands.

    Thus, $q$ is not an extreme ray, and our result is shown.
\end{proof}

\section{Thickened Graphs}

The class of complexes with extreme local rank 1 is of natural interest, since these are the nonchordal complexes where $\POS(\ASC)$ has the simplest extreme rays.
We can use theorem \ref{thm:small_chordal_quotient} to see that if $\ASC$ can be obtained from a chordal complex $\ASCb$ by identifying two vertices, then $\ASC$ will have extreme local rank 1.
In fact, there is a fairly rich class of complexes with extreme local rank 1, which can be obtained by `gluing together' chordal compexes in a controlled way.

To define our class of complexes, we start with a definition of a directed graph (possibly with loops and multi-edges), which is also sometimes called a quiver.
\begin{definition}[Directed Graph]\label{def:directed_graph}
    A directed graph $D$ consists of 2 sets, $V(D)$ and $E(D)$, and two functions $h : E(D) \rightarrow V(D)$, and $t : E(D) \rightarrow V(D)$.

    We think of $E(D)$ as being the set of edges of $D$, and $h$ as the map sending an edge to its head, and $t$ as the map sending an edge to its tail.
    Notice that this definition implicitly allows for multiple edges and loops.
\end{definition}
\begin{definition}[Thickened Graphs]\label{def:thick_graphs}
We now replace the edges of that graph with some arbitrary chordal complex with two distinguished vertices.
This construction takes in three inputs:
    \begin{enumerate}
        \item A directed graph $D$.
        \item For each edge in $E(D)$, a chordal complex $\ASC_e$.
        \item For each $\ASC_e$, two distinguished vertices, $H(e) \in \VER{\ASC}$ and $T(e) \in \VER{\ASC}$, so that $\{H(e), T(e)\}$ is not a face of $\ASC_e$.
    \end{enumerate}
    Given this data, we consider first the disjoint union (coproduct) of all of the $\ASC_e$ with the zero dimensional complex obtained by considering only the vertices of $V(D)$.
    \[
        \ASCb = \left(\bigsqcup_{e\in E(D)} \ASC_e\right) \bigsqcup V(D).
    \]

    Then, we define the \textbf{thickening} of $D$ by this data, $D / \{\ASC_e\}$ to be the quotient of $\ASCb$ obtained by identifying $h(e)$ with $H(e)$ and $t(e)$ with $T(e)$ for each $e \in E(D)$.
    \[
        \ASC = \ASCb / \{h(e) \sim H(e), t(e) \sim T(e)\}
    \]
    A \textbf{thickened graph} is any complex which can be obtained by thickening a graph by a family of chordal complex.
\end{definition}
Note that $\ASCb$ in the above definition is chordal, and the quotient map sending $\ASCb$ to $\ASC$ is a chordal cover.

Note in particular that any pure 1-dimensional complex can be taken as a thickened graph where the underlying directed graph is simple, and the chordal complexes are all single edge complexes.
\begin{thm}\label{thm:thickened_graphs_rank_1}
    Any thickened graph has extreme locally rank 1.
\end{thm}
\begin{proof}
    We will let $\ASC = D / \{\ASC_e\}$, and we will let $q \in \POS(\ASC)$ be extreme, with local rank greater than 1.

    While it is possible to argue in a similar way to the proof of theorem \ref{thm:small_chordal_quotient}, using the chordal cover $\phi : \ASCb \rightarrow \ASC$, it will be convenient to argue in a more coordinate-centered way here.

    For each $e \in E(D)$, let $\phi_e : \ASC_e \rightarrow \ASC$ be the natural inclusion.
    Each face of $\ASC$ is contained in some $\ASC_e$, so if for all $e \in E(D)$, $\phi_e^*(q)$ is locally rank 1, then in fact, $q$ must be locally rank 1.

    So, let $e \in E(D)$ be some edge so that $\phi^*_e(q)$ has local rank at least 2.
    We will first decompose $\phi^*_e(q)$ into a sum, and then extend this decomposition to a decomposition of $q$.
    Our decomposition for $\phi^*_e(q)$ comes from the next lemma (whose proof we defer momentarily).
    \begin{lemma}\label{lem:decomp2ver}
        Suppose that $X$ represents a quadratic in $\POS(\ASC_e)$, then there is some decomposition $X = Y+Z$ so that $Y, Z$ represent quadratic forms in $\POS(\ASC_e)$, $Y$ has rank 1, and for some $0 \le \alpha \le 1$, $Y_{T(e)T(e)} = \alpha X_{T(e)T(e)}$ and $Y_{H(e)H(e)} = \alpha X_{H(e)H(e)}$.
    \end{lemma}

    Let $X_e$ represent $\phi^*_e(q)$, and let $X_e = Y_e+Z_e$ denote the decomposition and $\alpha$ the constant given by lemma \ref{lem:decomp2ver}.

    Now, let $X$ be a partial matrix representing $q$.
    Let $Y$ be a partial matrix so that for  $\{i,j\} \in \ASC_e$, $Y_{i,j} = (Y_e)_{i,j}$, and for $\{i,j\} \not \in \ASC_e$, $Y_{i,j} = \alpha X_{i,j}$.
    Similarly, let $Z$ be a partial matrix so that for  $\{i,j\} \in \ASC_e$, $Z_{i,j} = (Z_e)_{i,j}$, and for $\{i,j\} \not \in \ASC_e$, $Z_{i,j} = (1-\alpha) X_{i,j}$.

    We see that $X = Y + Z$, since $X_{i,j} = Y_{i,j} + Z_{i,j}$ for each $i, j \in \VER{\ASC}$.

    We claim that $Y$ and $Z$ represent forms in $\POS(\ASC)$.
    To see this, it suffices to check that for each face $F \in \ASC$, the submatrix of $Y$ indexed by $F$, $Y|_F$ is PSD.

    If $F$ is contained in $\ASC_e$, then $Y|_F$ is PSD because it equals $Y_e|_F$, which is PSD, since $Y_e$ is PSD completable.
    On the other hand, if $F$ is not contained in $\ASC_e$, then $Y|_F = \alpha X|_F$, which is PSD since $\alpha \ge 0$ and $X|_F$ is PSD.
    Thus, $Y$ represents a form in $\POS(\ASC)$.

    By an identical argument $Z$ represents a form in $\POS(\ASC)$.

    Letting $q_1$ and $q_2$ denote the forms represented by $Y$ and $Z$ respectively, we see that $\phi^*_e(q_1)$ is rank 1, whereas $\phi^*_e(q)$ has rank greater than 1.

    $q = q_1 + q_2$, with $q_1, q_2 \in \POS(\ASC)$, but $q_1$ is not in the span of $q$, so this implies that $q$ does not span an extreme ray.

\end{proof}

We now prove lemma \ref{lem:decomp2ver}.
\begin{proof} (of lemma \ref{lem:decomp2ver})
    Because $\ASC_e$ is chordal, $X$ is PSD completable, we have that there is a PSD completion of $X$, $G(v_i : i \in \VER{{\ASC_e}})$, where the $v_i$ are in some Hilbert space $H$.

    If either $v_{T(e)}$ or $v_{H(e)}$ is 0, then choose $\unit$ to be any unit vector in the span of the $v_i$.
    If neither are 0, then consider $\frac{v_{T(e)}}{\|v_{T(e)}\|} + \frac{v_{H(e)}}{\|v_{H(e)}\|}$.
    If this is 0, then $v_{T(e)}$ and $v_{H(e)}$ are linearly dependent, and once again, we can simply choose $\unit$ to be any unit vector in the span of the $v_i$.

    Finally, if $\frac{v_{T(e)}}{\|v_{T(e)}\|} + \frac{v_{H(e)}}{\|v_{H(e)}\|}$ is nonzero, then let $\unit$ be a unit vector in the span of this vector.

    In any case, it is clear that for some $0 \le \alpha \le 1$,
    \[
        |\langle \unit, v_{T(e)} \rangle| =  \alpha \|v_{T(e)}\|,
    \]
    and
    \[ 
        |\langle \unit, v_{H(e)} \rangle| =  \alpha \|v_{H(e)}\|.
    \]
    
    Thus, if we let $T$ be the identity on $H$, and let $L = \SPAN \omega$, this gives us a $(T,L)$ decomposition of $X$ with the desired property.
\end{proof}
\section{Cones}
We recall that a cone vertex of a simplicial complex $\ASC$ is a vertex $v \in \VER{\ASC}$ so that $v$ is contained in all facets of $\ASC$.
Given a complex $\ASC$, the cone over $\ASC$ is defined as 
\[
    \hat{\ASC} = \{S \subseteq \groundSet \cup \{*\} : S \in \ASC \text{ or }S - * \in \ASC\}.
\]

\begin{thm}\label{thm:extremeRayCone}
    Every extreme ray of $q \in \POS(\hat{\ASC})$ is either sum-of-squares, or it satisfies $q(e_*) = 0$.
\end{thm}
\begin{proof}
    Let $q$ be an extreme ray of $\POS(\hat{\ASC})$, and let $X$ be a partial matrix representing $q$.

    If $X_{*,*} = 0$, then $q(e_*) = X_{*,*} = 0$, and we are done.
    So, assume that $X_{*,*} > 0$.

    Define $Z$ so that for $\{i,j\} \in \hat{\ASC}$.
    \[Z_{i,j} = X_{*,*}^{-1}X_{i,*} X_{*,j}.\]
    Notice that if we let $Z$ has a completion, $\hat{Z}$, where for all $i,j \in \VER{\hat{\ASC}}$,
    \[\hat{Z}_{i,j} = X_{*,*}^{-1}X_{i,*} X_{*,j}.\]
    Then $\hat{Z} = X_{*,*}^{-1}vv^{\intercal}$, where $v$ is the vector so that $v_i = X_{i,*}$.
    This is a rank 1 PSD matrix, so $Z$ represents a sum-of-squares form on $\POS(\hat{\ASC})$.

    We then define $Y = X - Z$.
    Specifically, $Y$ is a partial matrix so that for $\{i,j\} \in \skel 1 {\hat{\ASC}}$,
    \[Y_{i,j} = X_{i,j} - X_{*,*}^{-1}X_{i,*} X_{*,i}.\]
    It is not hard to see from this definition that $Y_{i,*} = 0$ for all $i \in \VER{\hat{\ASC}}$.

    We claim that $Y$ represents a quadratic in $\POS(\hat{\ASC})$.
    To check this, we need to check that $Y|_F$ is PSD for all faces of $\hat{\ASC}$.
    Firstly, if $F \in \ASC$, then $T = F \cup \{*\} \in \hat{\ASC}$, and we can consider the submatrix $X|_{T}$, which is PSD.
    The Schur complement of $X|_{T}$ with respect to the $*$ entry is then defined coordinate-wise as
    \[(X|_T / (*,*))_{i,j} = X_{i,j} -  X_{*,*}^{-1}X_{i,*} X_{*,j} = (Y|_F)_{i,j}.\]
    By lemma \ref{lem:schur_proj}, $(X|_T / (*,*))$ is PSD, and hence, $Y|_F$ is PSD.
    On the other hand, if $F \not \in \ASC$, then $F = T \cup \{*\}$ for some $T \in \ASC$, and we see then that $Y|_F$ is simply $Y|_T$ with an appended row and column of zeros.
    Since we have established that $Y|_T$ is PSD, this implies that $Y|_F$ is PSD.

    Notice that $Y + Z = X$. $X$ represents an extreme ray of $\POS(\ASC)$ which implies that $Y$ and $Z$ are linearly dependent, but $Y_{*,*} = 0 \neq Z_{*,*}$, so we must have that $Y$ is identically 0.
    Therefore, $X = Z$, and $X$ is rank 1 sum-of-squares.
\end{proof}
\section{Locally Rank 1 Extreme Rays}
In this section, we will consider complexes with extreme local rank 1, and classify the locally rank 1 extreme rays of $\POS(\ASC)$ for arbitrary complexes $\ASC$.

\subsection{Complexes of Extreme Local Rank 1}
Define
\[
    \EXSimp_1 = \{\ASC \in \simp : \elocr(\ASC) \le 1\}.
\]

From \ref{thm:thickened_graphs_rank_1}, we have that all thickened graphs are locally rank 1, and from \ref{thm:extremeRayCone}, the cone over a simplicial complex with extreme locally rank 1 has extreme local rank 1.
We will show one more operation preserves the class of locally rank 1 complexes.

Given two complexes, $\ASC$ and $\ASCb$ on disjoint vertex sets, and vertices $a \in \ASC$ and $b \in \ASCb$, then the 1-sum of $\ASC$ and $\ASCb$ is the complex obtained by identifying $a$ and $b$ in the disjoint union $\ASC \sqcup \ASCb$.
Denote this 1-sum by $\ASC \oplus \ASCb$ (where we intentionally suppress the choice of vertices in $\ASC$ and $\ASCb$).

\begin{thm}\label{thm:1sum}
    \[\elocr(\ASC \oplus \ASCb) = \max(\elocr(\ASC), \elocr(\ASCb))\]
\end{thm}
\begin{proof}
    Let $\phi_1 : \ASC \rightarrow \ASC \oplus \ASCb$ and $\phi_2 : \ASCb \rightarrow \ASC \oplus \ASCb$ denote the natural inclusion maps.
    Let $*$ denote vertex which was identified in $\ASC$ and $\ASCb$ to make $\ASC \oplus \ASCb$.
    
    Because every face of $\ASC \oplus \ASCb$ is either a face of $\ASC$ or $\ASCb$, if $q \in \CoorR_2(\ASC\oplus \ASCb)$, then $q \in \POS(\ASC)$ if and only if $\phi_1^*(q) \in \POS(\ASC)$ and $\phi_2^*(q) \in \POS(\ASC)$.
    By similar reasoning, $\locr(q) = \max(\locr(\phi_1^*(q)), \locr(\phi_2^*(q)))$.

    Moreover, given $q_1 \in \POS(\ASC)$ and $q_2 \in \POS(\ASCb)$ so that $q_1(e_*) = q_2(e_*)$, there is a unique $q \in \POS(\ASC)$ so that $\phi_1^*(q) = q_1$ and $\phi_2^*(q) = q_2$.
    To see this, if we let $Y$ be a partial matrix representing $q_1$ and $Z$ be a partial matrix representing $q_2$, then we can define a partial matrix $X$ so that
    \[
        X_{i,j} = Y_{i,j} \text{ if }\{i,j\} \in \ASC,\text{ and}
    \]
    \[
        X_{i,j} = Z_{i,j} \text{ if }\{i,j\} \in \ASCb.
    \]
    Note that there is no ambiguity here because if $\{i,j\} \in \ASC \cap \ASCb$, then $\{i,j\} = \{*\}$, and $Y_{*,*} = q_1(e_*) = q_2(e_*) = Z_{*,*}$.
    It can easily be seen that if $q$ is represented by $X$, then $q$ has the desired properties.

    Now, let $q$ be an extreme ray of $\POS(\ASC \oplus \ASCb)$.
    We can write $\phi_1^*(q)$ as the sum of linearly independent extreme rays of $\POS(\ASC)$, so that
    \[
        \phi_1^*(q) = r_1 + \dots + r_k,\text{ for extreme }r_1, \dots, r_k \in \POS(\ASC).
    \]

    Now, let $s_i = \frac{r_i(e_*)}{q(e_*)}\phi_2^*(q)$ for $i \in [k]$.
    Since we have that $s_i(e_*) = \frac{r_i(e_*)}{q(e_*)}\phi_2^*(q)(e_*) = r_i(e_*)$, we can define $q_i$ so that for each $i \in [k]$,
    \[
        \phi^*_1(q_i) = r_i, \text{ and}
    \]
    \[
        \phi^*_2(q_i) = s_i.
    \]

    If we sum the $q_i$, we see that
    \[
        \phi^*_1(\sum_{i = 1}^k q_i) = \sum_{i = 1}^k r_i = \phi^*_1(q), \text{ and}
    \]
    \[
        \phi^*_2(\sum_{i = 1}^k q_i) = \sum_{i = 1}^k s_i = \sum_{i = 1}^k \frac{r_i(e_*)}{q(e_*)}\phi_2^*(q) = \phi_2^*(q).
    \]
    Hence, $q = \sum_{i=1}^k q_i$.
    Since the $q_i$ are linearly independent, it is clear then that $k = 1$, and $\phi_1^*(q)$ spans an extreme ray of $\POS(\ASC)$.
    
    An identical argument shows that $\phi^*_2(q)$ spans an extreme ray of $\POS(\ASCb)$.

    Thus, both  $\phi_1^*(q)$ and $\phi_2^*(q)$ are extreme rays of $\POS(\ASC)$ and $\POS(\ASCb)$ respectively, and in particular, we can use the extreme local ranks of these complexes to bound the local ranks of these quadratics.

    Combining these observations, we obtain
    \[
        \locr(q) = \max(\locr(\phi_1^*(q)), \locr(\phi_2^*(q))) \le  \max(\elocr(\ASC), \elocr(\ASCb)).
    \]
\end{proof}

We can collect our existing results on locally rank 1 complexes in one theorem.
\begin{thm}\label{thm:sufficient_rank_1}
    $\EXSimp_1$ contains all thickened graphs, and is closed under the 1-sum and cone operations.
\end{thm}

\subsection{Classification of Locally Rank 1 Extreme Rays}
The locally rank 1 extreme rays of any complex $\ASC$ are of a fairly simple type.
We will see that in fact, that they are in some ways connected to the simplicial cohomology of $\ASC$ and its subcomplexes.

Firstly, it is worth pointing out that if $q$ is locally rank 1, and $\ASC$ is a connected complex, then $q$ is an extreme ray.
This follows from the fact that rank 1 PSD matrices are extreme inside the set of all PSD matrices.

Recall that if $q \in \POS(\ASC)$ is a complex, and $X$ is a partial matrix representing $q$, then the support of $q$ is 
\[
    \{i \in \VER{\ASC} : X_{ii} > 0\}
\]
Recall that $q \in \POS(\ASC)$ is of full support if $q(e_i) > 0 $ for all $i \in \VER{\ASC}$.
Recall that $q_1$ and $q_2$ are diagonally congruent if there is a diagonal linear transformation $D$ so that the induced map of $D : \mathcal{V}(\ASC) \rightarrow \mathcal{V}(\ASC)$ has the property that $\POS(D)(q_1) = q_2$.
If $D$ is such a diagonal linear transformation, we will use $D^*$ to denote the induced map $\POS(D)$.

Note that if $q$ is some quadratic represented by a partial matrix $X$, then $D^*(q)$ is represented by a matrix $D^*(X)$ so that
\[
    D^*(X)_{i,j} = D_{i,i}D_{j,j}X_{i,j}.
\]

Let $K_1$ denote the set of all orbits of locally rank 1 forms in $\POS(\ASC)$ of full support, treated as a semigroup under the Hadamard product.

We will refer to \cite{Hatcher} as a reference for the definition of simplicial cohomology.

\begin{thm}[Classification of Locally Rank 1 Extreme Rays]\label{thm:classify_lrk1_ext}
    The semigroup $K_1$ is isomorphic to the simplicial cohomology group $H^1(\ASC, \Z/2)$.
\end{thm}
\begin{proof}
    Let $q$ be a locally rank 1 extreme ray of $\POS(\ASC)$ with full support.

    We say that $q$ is normalized if $q(e_i) = 1$ for all $i \in \VER{\ASC}$.
    It is not hard to see that by performing a diagonal congruence, every $q \in \POS(\ASC)$ with full support is congruent to a normalized quadratic form in $\POS(\ASC)$.
    Also, it is not hard to see that normalized, locally rank 1 forms in $\POS(\ASC)$ form a semigroup under the Hadamard product.

    Let $X$ be a partial matrix representing a normalized $q$, so that $X_{ii} = 1$ for each $i \in \VER{\ASC}$.
    Moreover, for all $\{i,j\} \in \EDG{\ASC}$, we have that $X|_{\{i,j\}}$ is rank 1, so that
    \[
        X_{i,j}^2 = X_{ii} X_{jj} = 1.
    \]
    Therefore, $X_{i,j} \in \{-1, 1\}$. Say that $\sigma$ is the group isomorphism between $\{-1,1\}$ (as a multiplicative group) with $\Z / 2\Z$.
    
    We can assign to each normalized, locally rank 1, $q \in \POS(\ASC)$ an element of $\Z / 2\Z^{\EDG{\ASC}}$, namely, the function $f_X$ which assigns
    \[
        f_X(\{i,j\}) = \sigma(X_{i,j}).
    \]
    It is not hard to see that this defines a map of semigroups, since the Hadamard product is defined entry-wise.

    To construct our desired isomorphism, we need to show a few things: firstly, $f_X$ should be a cocycle, secondly, if $X_1$ and $X_2$ represent normalized, locally rank 1 forms in $\POS(\ASC)$, and $X_1$ and $X_2$ are diagonally congruent, then $f_{X_1}$ and $f_{X_2}$ should differ by a coboundary.
    Lastly, we need to show that the map sending $X$ to $f_X$ is in fact bijective.

    To check that $f_X$ is in fact a cocycle, we need to check that for any triangle $\{i,j,k\} \in \stskel 2 \ASC$, $f(\{i,j\}) + f(\{j,k\}) + f(\{i,k\}) = 0$, or equivalently that $X_{i,j}X_{j,k}X_{i,j} = 1$.
    To see this, simply note that the submatrix $X|_{\{i,j,k\}}$ is PSD, and its determinant is
    \[
        \det \begin{pmatrix} 1 & X_{i,j} & X_{i,k}\\ X_{i,j} & 1 & X_{j,k} \\ X_{i,k} & X_{j,k} & 1\end{pmatrix} = 1 + 2X_{i,j}X_{i,k}X_{j,k} - X_{i,j}^2 - X_{j,k}^2 - X_{i,k}^2 = 0.
    \]
    Noting that $X_{i,j}^2 = 1$ for all $\{i,j\} \in \stskel 1 \ASC$, we have that this implies that $X_{i,j}X_{i,k}X_{j,k} = 1$, as desired.

    On the other hand, suppose that $X_1$ and $X_2$ are diagonally congruent.
    Say that $D$ is this diagonal linear transformation so that $D^*(X_1) = D^*(X_2)$.

    We have that since $D^*(X_1)_{i,i} = D_{i,i}^2(X_1)_{i,i} = (X_2)_{i,i}$, we have that $D_{i,i}^2=1$, so that $D_{i,i} \in \{-1,1\}$.
    The equation $(X_1)_{i,j} = D_{i,i}D_{j,j}(X_2)_{i,j}$ then implies that
    \[
        f_{X_1}(\{i,j\}) = f_{X_2}(\{i,j\}) + \sigma(D_{i,i}) + \sigma(D_{j,j}).
    \]
    We see that function $g(\{i,j\}) = \sigma(D_{i,i}) + \sigma(D_{j,j})$ is a coboundary, so we have shown that $f$ in fact descends to a map from diagonal congruence classes to cohomology classes.

    Thus, the map sending $X$ to $f_X$ defines a semigroup homomorphism from $K_1$ to $H^1(\ASC, \Z / 2\Z)$.

    Finally, we need to show that this map is bijective. 
    Let $f \in \Z / 2\Z^{\EDG{\ASC}}$ be a coboundary.
    Let $X$ be the partial matrix so that $X_{i,i} = 1$ for $i \in \VER{\ASC}$, and for $\{i,j\} \in \EDG{\ASC}$,
    \[
        X_{i,j} = \sigma^{-1}(\{i,j\}).
    \]
    It is clear that this is well defined in the sense that if $f, g$ differ by a cocycle,  then the corresponding partial matrices are related by a diagonal congruence.

    To check that this represents a locally rank 1 quadratic form in $\POS(\ASC)$, we need to check that for all faces $F \in \ASC$, $X|_F$ is rank 1 and PSD.
    It is clear that for all $\{i,j\} \in \skel 1 \ASC$, $X|_{\{i,j\}}$ is rank 1 and PSD.
    On the other hand, if $\{i,j,k\} \in \stskel 2 \ASC$, then consider $X|_{\{i,j,k\}}$.
    To see that $X|_{\{i,j,k\}}$ is rank 1 and PSD, it suffices to show that all of the $2\times 2$ and $3\times 3$ minors are 0.
    It is clear that the $2\times 2$ minors of $X|_{\{i,j,k\}}$ are 0, and the determinant of $X|_{\{i,j,k\}}$ is 0 because of the cocycle condition.

    On the other hand, if $F$ is a general face of $\ASC$, then $X|_F$ is a symmetric matrix where all of the $2\times 2$ and $3\times 3$ minors are 0.
    Choosing some $i \in F$, and then by applying an appropriate diagonal congruence to $X|_F$, we may assume that $(X|_F)|_{i,j} = 1$ for all $j \in F$ without changing the signature of $X$.
    Since we have that for any $j,k \in F$, $X_{i,j}X_{i,k}X_{j,k} = 1$, so that $X_{j,k} = 1$, and thus, $X|_F$ is the all ones matrix (up to diagonal congruence), and so $X|_F$ is rank 1 and PSD.
\end{proof}

\section{Strongly Connected Maps}
Recall the definition of a strongly connected map.
\begin{definition}[Strongly Connected Map]\label{def:strong_connected_map}
We say a surjective map $\phi:\ASC \rightarrow \ASCb$ is \textbf{strongly connected} if it satisfies the following two properties:
\begin{enumerate}
    \item If $a, b \in \VER \ASC$, and $\phi(a) = \phi(b) = c \in \VER{{\ASCb}}$, then there is a sequence $a = x_1 ,x_2,x_3,\dots, x_k = b \in \VER \ASC$ so that $\phi(x_i)=c$ and $\{x_i, x_{i+1}\} \in \ASC$ for each $i \in [k-1]$.
    \item If $a, b \in \EDG{\ASC}$, and $\phi(a) = \phi(b) = c \in {\stskel 1 {\ASCb}}$, then there is a sequence $a = e_1 ,e_2,e_3,\dots, e_k = b \in \EDG{\ASC}$ so that $\phi(e_i) = c$, $e_i \cup e_{i+1} \in \ASC$, and $e_i \cap e_{i+1} \neq \varnothing$ for each $i \in [k-1]$.
\end{enumerate}
\end{definition}

\begin{thm}\label{thm:strong_conn}
    If $\phi : \ASC \rightarrow \ASCb$ is strongly connected, then the image of $\inphi$ is a face of $\POS(\ASC)$.
\end{thm}
\begin{proof}
    Let $q \in \POS(\ASC)$, and suppose that $\inphi(q) = q_1 + q_2$, for some $q_1, q_2 \in \POS(\ASC)$.
    Our goal is to show that $q_1$ must be in $\im \inphi$.

    Let $X, Y$ be partial matrices representing $\inphi(q)$ and $q_1$ respectively.

    Our goal will be to apply lemma \ref{lem:image_of_inphi}.
    For this, we would need to show that if $\{a,b\}, \{c,d\} \in \skel 1 {\ASC}$, and $\phi(\{a,b\}) = \phi(\{c,d\})$, then 
    \[
        Y_{a,b} = Y_{c, d}. \tag{1}
    \]
    The strongly connected hypothesis, together with the following two lemmas (proven afterwards) will imply equation 1 in the relevant cases.

    \begin{lemma}\label{lem:edge_equation}
        Suppose that $a,b \in \VER{{\ASC}}$, $\{a,b\} \in \ASC$, and $\phi(a) = \phi(b) \in \VER{{\ASCb}}$.
        Then $Y_{a,a} = Y_{a,b} = Y_{b,b}$.
    \end{lemma}
    \begin{lemma}\label{lem:triangle_equation}
        Suppose that $a,b,c \in \VER{{\ASC}}$, $\{a,b,c\} \in \ASC$, and $\phi(b) = \phi(c) \in \VER{{\ASCb}}$.
        Then, $Y_{a,b} = Y_{a,c}$.
    \end{lemma}
    
    To show that equation 1 holds whenever $\phi(\{a,b\}) = \phi(\{c,d\})$, we divide this into the cases that $\phi(\{a,b\}) = \phi(\{c,d\}) \in \VER{{{\ASC}}}$, and $\phi(\{a,b\}) = \phi(\{c,d\}) \in \EDG{{{\ASC}}}$.

    \textbf{Case 1: }Suppose that $\phi(\{a,b\}) = \phi(\{c,d\}) \in \VER{{{\ASC}}}$.
    Note that this implies that $\phi(a) = \phi(b)$, so that $X_{a,a} = X_{a,b} = X_{b,b}$ by lemma \ref{lem:edge_equation}, and identically, $X_{c,c} = X_{c,d} = X_{d,d}$.
    So, it remains to show that $X_{a,a} = X_{c,c}$ to complete this chain of equations.

    By property (1) of being strongly connected, there is a sequence 
    \[
        a = x_1, x_2, \dots, x_k = c,
    \]
    so that $\{x_i, x_{i+1}\} \in \ASC$ and $\phi(x_i) = \phi(a) = \phi(c) \in \VER{{\ASCb}}$, for each $i \in [k-1]$.
    Applying lemma \ref{lem:edge_equation}, we have that for each $i$,
    \[
        Y_{x_i, x_{i}} = Y_{x_{i}, x_{i+1}} = Y_{x_{i+1}, x_{i+1}}.
    \]
    Chaining these equalities together, we obtain that $Y_{a,b} = Y_{c,d}$, satisfying this case for lemma \ref{lem:image_of_inphi}.

    \textbf{Case 2: }Suppose that $\phi(\{a,b\}) = \phi(\{c,d\}) \in \EDG{{\ASCb}}$.
    By property 2 of being strongly connected, there is a sequence
    \[
        \{a,b\} = e_1, e_2, \dots, e_k = \{c,d\},
    \]
    so that $\phi(e_i) = \phi(\{a,b\}) = \phi(\{c,d\})$, $e_i \cup e_{i+1} \in \ASC$, and $e_i \cap e_{i+1} \neq \varnothing$.

    We see then that $e_i \cup e_{i+1} \in \skel 3 \ASC$, and so there is some $\{x_i, y_i, z_i\} \in \ASC$ so that $e_i = \{x_i, y_i\}$ and $e_{i+1} = \{x_i, z_i\}$.
    The fact that $\phi(e_i) = \phi(e_{i+1})$ then implies that $\phi(y_i) = \phi(z_i)$, so we can apply lemma  \ref{lem:triangle_equation} to assert that 
    \[
        Y_{x_i, y_i} = Y_{x_i, z_i}. 
    \]
    Since $\{a,b\} = \{x_1, y_1\}$ and $\{c,d\} = \{x_{k-1}, z_{k-1}\}$, by chaining these equations together from $i = 1$ up to $i = k-1$, we obtain that
    \[
        Y_{a,b} = Y_{c,d}.
    \]
    As we desired.

    We see then that $Y$ represents a form in $\im \inphi$, and an identical argument shows that $Z$ also represents a form in $\im \inphi$.
    Therefore, $\im \inphi$ is a face of $\POS(\ASC)$.
\end{proof}

\begin{proof}(of lemma \ref{lem:edge_equation})

    Recall our state, where $X$ represents a quadratic form $\inphi(q)$, and $X = Y + Z$, where $Y, Z$ both represent forms in $\POS(\ASC)$.

    Let $A$ be a partial matrix representing $q$, so that for $\{a,b\} \in \ASC$,
    \[
        X_{a,b} = A_{\phi(a), \phi(b)}.
    \]

    If we have that $\phi(\{a,b\}) = i \in \VER{\ASC}$, then
    \[
        X|_{\{a,b\}} =
        \begin{pmatrix}
            X_{a,a} & X_{a,b}\\
            X_{a,b} & X_{b,b}
        \end{pmatrix} =
        \begin{pmatrix}
            A_{\phi(a),\phi(a)} & A_{\phi(a),\phi(b)}\\
            A_{\phi(a),\phi(b)} & A_{\phi(b),\phi(b)}
        \end{pmatrix} = 
        \begin{pmatrix}
            A_{i,i} & A_{i,i}\\
            A_{i,i} & A_{i,i}
        \end{pmatrix}.
    \]
    We see that $X|_{\{a,b\}}$ is rank 1, so $X|_{\{a,b\}}$ is extreme in the cone of $2\times 2$ PSD matrices.
    Since $X = Y + Z$, $X|_{\{a,b\}} = Y|_{\{a,b\}} + Z|_{\{a,b\}}$, and since $Y|_{\{a,b\}}$ is PSD, we have that $Y|_{\{a,b\}}$ is a multiple of $X|_{\{a,b\}}$.
    In particular, since $X_{a,a} = X_{a,b} = X_{b,b}$, we have that $Y_{a,a} = Y_{a,b} = Y_{b,b}$.
\end{proof}

\begin{proof}(of lemma \ref{lem:triangle_equation})

    Recall our state, where $X$ represents a quadratic form $\inphi(q)$, and $X = Y + Z$, where $Y, Z$ both represent forms in $\POS(\ASC)$.

    Let $A$ be a partial matrix representing $q$, so that for $\{a,b\} \in \ASC$,
    \[
        X_{a,b} = A_{\phi(a), \phi(b)}.
    \]

    For $\{a,b,c\} \in \ASC$ so that $\phi(a) = e$, and $\phi(b) = \phi(c) = f$, we have that
    \[
        X|_{\{a,b,c\}} =
        \begin{pmatrix}
            X_{a,a} & X_{a,b} & X_{a,c}\\
            X_{a,b} & X_{b,b} & X_{b,c}\\
            X_{a,c} & X_{b,c} & X_{c,c}\\
        \end{pmatrix} = 
        \begin{pmatrix}
            A_{e,e} & A_{e,f} & A_{e,f}\\
            A_{e,f} & A_{f,f} & A_{f,f}\\
            A_{e,f} & A_{f,f} & A_{f,f}
        \end{pmatrix}.
    \]
    Notice that $X|_{\{b,c\}}$ is rank 1, and since $X|_{\{b,c\}} = Y|_{\{b,c\}} + Z|_{\{b,c\}}$, $Y|_{\{b,c\}}$ is a multiple of $X|_{\{b,c\}}$.

    Now, we know that $Y|_{\{a,b,c\}}$ is a $3\times 3$ PSD matrix, where $Y_{b,b}=Y_{b,c} = Y_{c,c}$, so that we can write
    \[
        Y|_{\{a,b,c\}} =
        \begin{pmatrix}
            Y_{a,a} & Y_{a,b} & Y_{a,c}\\
            Y_{a,b} & Y_{b,b} & Y_{b,b}\\
            Y_{a,c} & Y_{b,b} & Y_{b,b}
        \end{pmatrix},
    \]

    To see that this implies that $Y_{a,b} = Y_{a,c}$, use the Gram matrix characterization so that $Y|_{\{a,b,c\}} = G(v_a, v_b, v_c)$, then the fact that $Y_{b,b} = Y_{b,c} = Y_{c,c}$ implies that $\|v_b\|^2 = \langle v_b, v_c\rangle = \|v_c\|^2$, which by Cauchy-Schwarz implies that $v_b = v_c$. 
    In particular, $Y_{a,b} = \langle v_a, v_b\rangle = \langle v_a, v_c \rangle = Y_{a,c}$, as was desired.

\end{proof}

Let us make a few remarks about this theorem.
Because we assumed that $\phi$ is surjective in the definition of strongly connected maps, $\phi^*$ is injective.
Therefore, if $q$ spans an extreme ray of $\POS(\ASCb)$, and $\phi : \ASC \rightarrow \ASCb$ is strongly connected, then $\phi^*(q)$ spans an extreme ray.
Also, since $\phi$ is surjective, $\phi^*$ preserves local rank, so that if $q$ is an extreme ray of $\POS(\ASCb)$ with local rank $\elocr(\ASCb)$, $\phi^*(q)$ spans an extreme ray of $\POS(\ASC)$ so that $\locr(\phi^*(q)) = \elocr(\ASCb)$.

\begin{cor}
    If $\phi : \ASC \rightarrow \ASCb$ is a strongly connected map, then $\elocr(\ASC) \ge \elocr(\ASC)$.
\end{cor}

The idea for this theorem stems from a proof of Laurent concerning the way that Gram dimension interacts with edge contraction, found in \cite{laurent2012gram}.
Though the Laurent's proof neither is implied by nor implies our results, it has a spiritually connected to our results by the following lemma (which we do not prove).
\begin{lemma}
    Let $G$ be a graph.
    Let $e$ be an edge of $G$ which is contained in no 4-cycles, and let $ G / e$ be the result of contracting the edge $e$ in $G$.
    $\chi(G)$ and $\chi(G/e)$ are the clique complexes of $G$ and $G/e$ respectively. 
    Then the natural quotient map $\phi^* : \chi(G) \rightarrow \chi(G/e)$ is strongly connected.
\end{lemma}

We will also see in our next section that this idea has some connections to algebraic topology.
\section{Maps to Spheres and Combinatorial Manifolds}
\begin{thm}[Maps to Spheres]\label{thm:map_to_sphere}
    Suppose that $\ASC$ is a simplicial complex, and $F$ is a facet of $\ASC$ of size $d+1 > 1$ with the following properties.
    (Use $F^c$ to denote $E-F$.)
    \begin{enumerate}
        \item For any $a, b \in F^c$, there is a sequence $x_1 ,\dots, x_k \in F^c$ so that $x_1 = a$, $x_k = b$ and $\{x_i, x_{i+1}\} \in \ASC$ for each $i \in [k-1]$.
        \item For each $c \in F$, and for any $a, b \in F^c$ so that $\{a,c\}, \{b,c\} \in \ASC$, there is a sequence $x_1 ,\dots, x_k \in S^c$ so that $x_1 = a$, $x_k = b$ and $\{x_i, x_{i+1}, c\} \in \ASC$ for each $i \in [k-1]$.
        \item For each $c \in F$, there is some $w \in F^c$ so that $F - c + w \in \ASC$.
    \end{enumerate}

    Then there is strongly connected map from $\ASC$ to $S_d$.
\end{thm}
\begin{proof}
    Enumerate the elements of $F = \{x_1 ,\dots, x_{d+1}\}$.
    We consider the map 
    \[
        \phi : \ASC \rightarrow S_d
    \]
    which sends $x_i$ to $i \in \VER{S_d}$, and for any $x \not \in F$, we send $x$ to $d+2 \in \VER{S_d}$.

    To see that this map is simplicial, we need only check that for all $T \in \ASC$, $|\phi(T)| \le d+1$.
    This follows because $F$ is a facet; if we had that $|\phi(T)| = d+2$, then $T$ would have to strictly contain $F$.

    It is also surjective, because facets of $S_d$ are precisely the sets of the form $[d+2] - i$, for some $i \in [d+2]$.
    Now, note that if $i = d+2$, then $\phi(F) = [d+2] - i$, and for any other $i$, there exists some $w \in F^c$ so that $F - x_i + w \in \ASC$, and $\phi(F- x_i + w) = [d+2] - i$.

    We now need to check that $\phi$ satisfies properties (1) and (2) in the definition of strong connectivity.

    Firstly, note that if $a, b \in \VER{\ASC}$ where $a\neq b$, then $\phi(a) = \phi(b)$ if and only if $a, b \not \in F$.
    Then by property (2) of $\ASC$, we have that there is a sequence $a = x_1, \dots, x_k = b$ so that $\{x_i, x_{i+1}\} \in \ASC$, and $\phi(x_i) = \phi(b)$ for each $i \in [k-1]$.
    The existence of such path implies that the map $\phi$ satisfies property (1) of being strongly connected.

    Secondly, if we see that if $e, f \in \EDG{\ASC}$, then $\phi(e) = \phi(f)$ if and only if there is some $a,b,c$ so that $c \in F$, and $a, b \not \in F$, so that $e = \{a,c\}$ and $f = \{b,c\}$.
    In this case, by property (2) of the face $F$, we have a sequence $a = x_1, \dots, x_k = b$ so that $\{x_i, x_{i+1}, c\} \in \ASC$ for each $i \in [k-1]$.
    Again, this can easily be seen to imply that $\phi$ satisfies property (2) of being strongly connected.
\end{proof}

Recall that a combinatorial manifold of dimension $d$ is a connected simplicial complex where for each $i \in \VER{\ASC}$, the link $\lnk_i(\ASC)$ has a topological realization which is homemomorphic to a sphere of dimension $d-1$, and also $|\ASC|$ is homeomorphic to a closed topological manifold.
The following proof was inspired by a Math StackExchange comment\cite{stackoverflow_comment}.

Because our main goal is to highlight the usage of theorem \ref{thm:map_to_sphere}, we will not be entirely rigorous in our treatment of the topological aspects of the proof.
In particular, we will require the following topological fact, which we will assert without proof.
\begin{lemma}\label{lem:topological}
    If $M$ is a closed, connected manifold, and $F$ is a closed subset of $M$ which is homeomorphic to a Euclidean ball, then the set $M - F$ is connected.
\end{lemma}
\begin{thm}
    If $\ASC$ is a combinatorial manifold, then there is a strongly connected map from $\ASC$ to $S_n$.
\end{thm}
\begin{proof}
    Fix any $F \in \ASC$. We want to show that $F$ satisfies the conditions of theorem \ref{thm:map_to_sphere}, which will imply the existence of the desired map.

    To show property (1), we let $M = |\ASC|$ be the topological realization of $\ASC$, and note that $F \subseteq M$ is closed and homeomorphic to a ball, so that by lemma \ref{lem:topological}, we have that $M - F$ is connected.
    Therefore, we have that $M - F$ is actually path connected, so that for any $a, b \in \VER{\ASC} - F$, there is a continuous path from $a$ to $b$ in $M$, $p$.
    Taking a simplicial approximation to $p$, we obtain a sequence of vertices in $\ASC$ so that $x_1, \dots, x_k \in F^c$ so that $x_1 = a$, $x_k = b$, and $\{x_i, x_{i+1}\} \in \ASC$, as desired.

    To show property (2), we can apply our assumption that for any $c \in \VER{\ASC}$, $N = |\lnk_c(\ASC)|$ is a topological sphere.
    Applying this to $c \in F$, we notice that $F-c$ is a face of $\lnk_c(\ASC)$.
    Once again applying lemma \ref{lem:topological} we see that $N - (F-c)$ is connected, and therefore, for any $a, b \in \VER{\ASC} - (F-c)$, we obtain a continuous path from $a$ to $b$ in $N - (F-c)$.
    Once again, taking a simplicial approximation of $p$ gives us a sequence of vertices $a = x_1, \dots, x_k = b \in \VER{\lnk_c(\ASC)} - (F-c)$ so that $\{x_i, x_{i+1}\} \in \lnk_c(\ASC)$.
    By definition of the link, this implies that $\{x_i, x_{i+1}, c\} \in \ASC$.
    Thus, for any $\{a,c\}, \{b,c\} \in \ASC$, where $a,b \in F^c$, we obtain $a = x_1, \dots, x_k = b \in F^c$ so that $\{x_i, x_{i+1}, c\} \in \ASC$ for all $i \in [k-1]$, showing property (2).

    To show property (3), let $c \in F$, and suppose that there are no vertices $w \in F^c$ so that $F-c + w \in \ASC$. 
    We see then that if we take a point in $x$ in the relative interior of the face $F - c \subseteq M$, then every sufficiently small neighborhood of $x$ which is homeomorphic to a half-space, and in particular, is not homeomorphic to a Euclidean ball.
    This contradicts the fact that $M$ is a topological manifold.
\end{proof}

\section{Acknowledgements}
We would like to acknowledge Greg Blekherman, Justin Chen and Daniel Minahan at the Georgia Institute of Technology for productive conversations.

\bibliographystyle{plain}
\bibliography{main}
\end{document}